\documentclass{birkjour}
 \newtheorem{thm}{Theorem}[section]
 
 \newtheorem{lem}[thm]{Lemma}
 
 \theoremstyle{definition}
 \newtheorem{defn}[thm]{Definition}
 \theoremstyle{remark}
 \newtheorem{rem}[thm]{Remark}
 
 \numberwithin{equation}{section}

 \usepackage{amssymb}

\usepackage{graphicx}
\usepackage{latexsym}
\usepackage{amsmath}
\usepackage{amssymb}
\usepackage[colorlinks]{hyperref}
\usepackage{paralist}

\newcommand{\comment}[1]{}
\allowdisplaybreaks
\begin{document}

%
%
%
%
%
%
%
%
%

\title[Newtonian and single layer potentials for the Stokes system]
{Newtonian and single layer potentials for the Stokes system with $L^{\infty}$ coefficients and the exterior Dirichlet problem}

\author[M. Kohr]{Mirela Kohr}
\address{%
Faculty of Mathematics and Computer Science, Babe\c{s}-Bolyai University,\\
1 M. Kog\u alniceanu Str.,
400084 Cluj-Napoca, Romania}

\email{mkohr@math.ubbcluj.ro}

\thanks{The authors acknowledge the support of the grant EP/M013545/1: "Mathematical Analysis of Boundary-Domain Integral Equations for Nonlinear PDEs" from the EPSRC, UK}

\author[S. E. Mikhailov]{Sergey E. Mikhailov}
\address{Department of Mathematics, Brunel University London,\\
             Uxbridge, UB8 3PH, United Kingdom}

\email{sergey.mikhailov@brunel.ac.uk}

\author[W. L. Wendland]{Wolfgang L. Wendland}
\address{Institut f\"ur Angewandte Analysis und Numerische Simulation, Universit\"at Stuttgart,\\
Pfaffenwaldring, 57,
70569 Stuttgart, Germany}

\email{wendland@mathematik.uni-stuttgart.de}

\thanks{}



\date{January 1, 2004}
\dedicatory{Dedicated to Professor H. Begehr on the occasion of his $80$th birthday.}

\begin{abstract}
A mixed variational formulation of some problems in $L^2$-based Sobolev spaces is used to define the Newtonian and layer potentials for the Stokes system with $L^{\infty}$ coefficients on Lipschitz domains in ${\mathbb R}^3$.
Then the solution of the exterior Dirichlet problem for the Stokes system with $L^{\infty}$ coefficients is presented in terms of these potentials and the inverse of the corresponding single layer operator.
\end{abstract}

\maketitle
\section{Introduction}
Let 
${\bf u}$ be an unknown vector field, $\pi $ be an unknown scalar field, and ${\bf f}$ be a given vector field defined on an exterior Lipschitz domain $\Omega_{-}\subset {\mathbb R}^3$. Let also ${\mathbb E}({\bf u})$ be the symmetric part of the gradient of ${\bf u}$, $\nabla {\bf u}$. Then the equations
\begin{equation}
\label{Stokes}
\begin{array}{lll}
{\mathcal L}_{\mu }({\bf u},\pi ):={\rm{div}}\left(2\mu {\mathbb E}({\bf u})\right)-\nabla \pi={\bf f},\ {\rm{div}}\ {\bf u}=0 \mbox{ in } \Omega _{-}
\end{array}
\end{equation}
determine the {\it Stokes system} with a known viscosity coefficient $\mu\in L^{\infty}(\Omega_-)$.
This linear PDE system describes the flows of viscous incompressible fluids, when the inertia of such a fluid can be neglected. The {coefficient $\mu$ is} related to the physical properties of the fluid (for further details we refer the reader to the books \cite{Ni-Be} and \cite{Galdi} and the references therein).

The methods of layer potential theory have a main role in the analysis of boundary value problems for elliptic partial differential equations (see, e.g., \cite{Ch-Mi-Na,Co,H-W,12,Med-CVEE,M-W,24}). Fabes, Kenig and Verchota \cite{Fa-Ke-Ve} obtained mapping properties of layer potentials for the constant coefficient Stokes system in $L^p$ spaces.
Mitrea and Wright \cite{M-W} have used various methods of layer potentials in the analysis of the main boundary problems for the Stokes system with constant coefficients in arbitrary Lipschitz domains in ${\mathbb R}^n$. The authors in \cite{K-L-M-W} have obtained mapping properties of the constant coefficient Stokes layer potential operators in standard and weighted Sobolev spaces by exploiting results of singular integral operators. Gatica and Wendland \cite{Ga-We} used the coupling of mixed finite element and boundary integral methods for solving a class of linear and nonlinear elliptic boundary value problems. The authors in \cite{K-L-W} used the Stokes and Brinkman integral layer potentials and a fixed point theorem to show an existence result for a nonlinear Neumann-transmission problem for the Stokes and Brinkman systems with data in $L^p$, Sobolev, and Besov spaces (see also \cite{K-M-W4}). All above results are devoted to elliptic boundary value problems with constant coefficients.

Potential theory plays also a main role in the study of elliptic boundary value problems with variable coefficients. Dindo\u{s} and Mitrea \cite{D-M} have obtained well-posedness results in Sobolev spaces for Poisson problems for the Stokes and Navier-Stokes systems with Dirichlet condition on $C^1$ and Lipschitz domains in compact Riemannian manifolds by using mapping properties of Stokes layer potentials in Sobolev and Besov spaces. 
Chkadua, Mikhailov and Natroshvili \cite{Ch-Mi-Na-3} obtained direct segregated systems of boundary-domain integral equations for a mixed boundary value problem of Dirichlet-Neumann type for a scalar second-order divergent elliptic partial differential equation with a variable coefficient in an exterior domain of ${\mathbb R}^3$ (see also \cite{Ch-Mi-Na}). Hofmann, Mitrea and Morris \cite{H-M-M} considered layer potentials in $L^p$ spaces for elliptic operators of the form $L\!=\!-{\rm{div}}(A\nabla u)$ acting in the upper half-space ${\mathbb R}_{+}^{n}$, $n\geq 3$, or in more general Lipschitz graph domains, with an $L^{\infty }$ coefficient matrix $A$, which is $t$-independent, and solutions of the equation $L u \!=\!0$ satisfy interior De Giorgi-Nash-Moser estimates. They obtained a Calder\'{o}n-Zygmund type theory associated to the layer potentials, and well-posedness results of boundary problems for the operator $L$ in $L^p$ and endpoint spaces.

Our variational approach is inspired by that developed by Sayas and Selgas in \cite{Sa-Se} for the constant coefficient Stokes layer potentials on Lipschitz boundaries, and is based on the technique of N\'{e}d\'{e}lec \cite{Ne}. Girault and Sequeira \cite{Gi-Se} obtained a well-posed result in weighted Sobolev spaces for the Dirichlet problem for the standard Stokes system in exterior Lipschitz domains of ${\mathbb R}^n$, $n=2,3$. B\u{a}cu\c{t}\u{a}, Hassell and Hsiao \cite{B-H} developed a variational approach for the standard Brinkman single layer potential and used it in the analysis of the time dependent exterior Stokes problem with Dirichlet boundary condition in ${\mathbb R}^n$, $n=2,3$. Barton \cite{Barton} constructed layer potentials for {strongly} elliptic differential operators in general settings by using the Lax-Milgram theorem, and generalized various properties of layer potentials for harmonic and second order elliptic equations.
Brewster et al. in \cite{B-M-M-M} have used a variational approach and a deep analysis to obtain well-posedness results for boundary problems of Dirichlet, Neumann and mixed type for higher order divergence-form elliptic equations with $L^{\infty }$ coefficients in locally $(\epsilon,\delta )$-domains and in Besov and Bessel potential spaces.
Choi and Lee \cite{Choi-Lee} have studied the Dirichlet problem for the Stokes system with nonsmooth coefficients, and proved the unique solvability of the problem in Sobolev spaces on a bounded Lipschitz domain $\Omega \subset {\mathbb R}^n$ ($n\geq 3$) with a small Lipschitz constant when the coefficients have vanishing mean oscillations with respect to all variables. Choi and Yang \cite{Choi-Yang} obtained the existence and pointwise bound of the fundamental solution for the Stokes system with measurable coefficients in ${\mathbb R}^n$, $n\geq 3$, whenever the weak solutions of the system are locally H\"{o}lder continuous. Alliot and Amrouche \cite{Al-Am} have used a variational approach to obtain weak solutions for the exterior Stokes problem in weighted Sobolev spaces. Also, Amrouche and Nguyen \cite{Amrouche-1} proved existence and uniqueness results in weighted Sobolev spaces for the Poisson problem with Dirichlet boundary condition for the Navier-Stokes system in exterior Lipschitz domains in ${\mathbb R}^3$.

The purpose of this work is to show the well-posedness result of the Poisson problem of Dirichlet type for the Stokes system with $L^{\infty}$ coefficients in $L^2$-based Sobolev spaces on an exterior Lipschitz domain in ${\mathbb R}^3$. We use a variational approach that reduces this boundary value problem to a mixed variational formulation. A similar variational approach is used to define the Newtonian and layer potentials for the Stokes system with $L^{\infty}$ coefficients on Lipschitz surfaces in ${\mathbb R}^3$, by using the weak solutions of some transmission problems in $L^2$-based Sobolev spaces. Finally, the mapping properties of these layer potentials are used to construct explicitly the solution of the exterior Dirichlet problem for the Stokes system with $L^{\infty}$ coefficients.
The analysis developed in this paper confines to the case $n=3$, due to its practical interest, but the extension to the case $n\geq 3$ can be done with similar arguments.

\section{Functional setting and useful results}
\setcounter{equation}{0}

Let $\Omega _{+}:=\Omega \subset {\mathbb R}^3$ be a bounded Lipschitz domain, i.e., an open connected set whose boundary $\partial {\Omega }$ is locally the graph of a Lipschitz function. Assume that $\partial \Omega $ is connected. Let $\Omega _{-}:={\mathbb R}^3\setminus \overline{{\Omega }}_{+}$ denote the exterior Lipschitz domain. Let $\mathring E_\pm$ denote the operators of extension by zero outside $\Omega_\pm$.

\subsection{\bf Standard $L^2$-based Sobolev spaces and related results}
{Let} ${\mathcal F}$ and ${\mathcal F}^{-1}$ denote the Fourier transform and its inverse defined on the
$L^1(\mathbb R^3)$ functions 
and generalized to the space of tempered distributions ${\mathcal S}^*({\mathbb R}^3)$ (i.e., the topological dual of the space ${\mathcal S}({\mathbb R}^3)$ of all rapidly decreasing infinitely differentiable functions on ${\mathbb R}^3$). The Lebesgue space of (equivalence classes of) measurable, square integrable functions on ${\mathbb R}^3$ is denoted by $L^2({\mathbb R}^3)$, and by $L^{\infty }({\mathbb R}^3)$ we denote the space of (equivalence
classes of) essentially bounded measurable functions on ${\mathbb R}^3$.
Let $H^1({\mathbb R}^3)$ and $H^1({\mathbb R}^3)^3$ denote the $L^2$-based Sobolev (Bessel potential) spaces
\begin{align}
\label{bessel-potential}
&\!H^1({\mathbb R}^3)\!:=\!\big\{f\!\in \!{\mathcal S}^*({\mathbb R}^3):\|f\|_{H^1({\mathbb R}^3)}\!=\!\big\|{\mathcal F}^{-1}[(1\!+\!|\xi |^2)^{\frac{1}{2}}{\mathcal F}f]\big\|_{L^2({\mathbb R}^3)}\!<\!\infty \!\big\},
\\
&\!H^1({\mathbb R}^3)^3:=\{f=(f_1,f_2,f_3):f_j\in H^1({\mathbb R}^3),\ j=1,2,3\}
\label{bessel-potential2}
\end{align}
Now let $\Omega '$ be $\Omega_+$, $\Omega_-$ or $\mathbb R^3$. We denote by 
${\mathcal D}(\Omega '):=C^{\infty }_{0}(\Omega ')$ the space of infinitely differentiable functions with compact support in $\Omega '$, equipped with the inductive limit topology. 
Let 
${\mathcal D}^*(\Omega ')$ denote the corresponding space of distributions on $\Omega '$, i.e., the dual space of ${\mathcal D}(\Omega ')$.
Let us consider the space
\begin{align}
\label{spaces-Sobolev-inverse}
&H^1({\Omega '}):=\{f\in {\mathcal D}^*(\Omega '):\exists \ F\in H^1({\mathbb R}^3)
\mbox{ such that } F|_{\Omega '}=f\}\,,
\end{align}
where $\cdot |_{\Omega '}$ is the restriction operator to $\Omega '$.
The space $\widetilde{H}^1({\Omega '})$ is the closure of ${\mathcal D}(\Omega ')$ in $H^1({\mathbb R}^3)$. This space can be also characterized as
\begin{align}\label{2.5}
&\widetilde{H}^1({\Omega '}):=\left\{\widetilde{f}\in H^1({\mathbb R}^3):{\rm{supp}}\,
\widetilde{f}\subseteq \overline{\Omega '}\right\}.
\end{align}
Similar to definition \eqref{bessel-potential2}, $H^1({\Omega '})^3$ and $\widetilde{H}^1({\Omega '})^3$ are the spaces of vector-valued functions whose components belong to the scalar spaces {$H^1({\Omega '})$ and $\widetilde{H}^1({\Omega '})$}, respectively (see, e.g., \cite{Lean}). 
The Sobolev space $\widetilde{H}^1({\Omega '})$ can be identified with the closure $\mathring{H}^{1}(\Omega ')$ of ${\mathcal D}(\Omega ')$ in the norm of ${H}^1({\Omega '})$ (see, e.g., \cite[(3.11), (3.13)]{M-M-W}, \cite[Theorem 3.33]{Lean}). The space ${\mathcal D}(\overline{\Omega '})$ is dense in ${H}^1({\Omega '})${,} and the following spaces can be isomorphically identified
({cf., e.g.,} \cite[Theorem 3.14]{Lean})
\begin{equation}
\label{duality-spaces} ({H}^1({\Omega '})){^*}=\widetilde{H}^{-1}({\Omega '}),\
\ {H}^{-1}({\Omega '})=(\widetilde{H}^1({\Omega '})){^*}.
\end{equation}
For $s\in [0,1]$, the {Sobolev space ${H}^s(\partial\Omega )$ on the boundary $\partial\Omega$} can be defined by using the space ${H}^s({\mathbb R}^{2})$, a partition of unity and the pull-backs of the local parametrization of $\partial \Omega $, and ${H}^{-s}(\partial\Omega  )=\left({H}^s(\partial\Omega )\right)^*.$ All the above spaces are Hilbert spaces. For further properties of Sobolev spaces on bounded Lipschitz domains and Lipschitz boundaries, we refer to \cite{Agr-1,J-K1,Lean,M-W,Triebel}.

A useful result for the next arguments is given below (see, e.g., \cite{Co}, \cite[Proposition 3.3]{J-K1}). 
\begin{lem}
\label{trace-operator1} Assume that ${\Omega }:=\Omega _{+}\subset {\mathbb R}^3$ is a bounded Lipschitz domain with connected boundary $\partial \Omega $ and denote by $\Omega _{-}:={\mathbb R}^3\setminus \overline{\Omega }$ the corresponding exterior domain. Then there {exist linear and bounded trace operators} $\gamma_{\pm }:{H}^{1}({\Omega }_{\pm })\to H^{\frac{1}{2}}({\partial\Omega  })$ such that $\gamma_{\pm }f=f|_{{{\partial\Omega }}}$ for any $f\in C^{\infty }(\overline{\Omega }_{\pm })$.
{These operators are surjective and have $($non-unique$)$ bounded linear right inverse operators} $\gamma^{-1}_{\pm }:H^{\frac{1}{2}}({\partial\Omega })\to {H}^{1}({\Omega }_{\pm }).$
\end{lem}

The jump of a function $u\in H^1({\mathbb R}^3\setminus \partial \Omega )$ across $\partial \Omega $ is denoted by $\left[{\gamma }(u)\right]:={\gamma }_{+}(u)-{\gamma }_{-}(u)$.
For $u\in H^1_{\rm loc}({\mathbb R}^3)$, $\left[{\gamma }(u)\right]=0$. The trace operator $\gamma :H^1({\mathbb R}^3)\to H^{\frac{1}{2}}(\partial \Omega )$ can be also considered and is linear and bounded\footnote{The trace operators defined on Sobolev spaces of vector fields on ${\Omega }_{\pm }$ or ${\mathbb R}^3$ are also denoted by $\gamma_{\pm }$ and $\gamma $, respectively.}.

If $X$ is either an open subset or a surface in ${\mathbb R}^3$, then we use the notation $\langle \cdot ,\cdot \rangle _X$ for the duality pairing of two dual Sobolev spaces defined on $X$.

\subsection{\bf Some weighted Sobolev spaces and related results}
\label{S2.2}
For a point ${\bf x}=(x_1,x_2,x_3)\in {\mathbb R}^3$, its distance to the origin is denoted by $|{\bf x}|=(x_1^2+x_2^2+x_3^2)^{\frac{1}{2}}$. Let $\rho $ denote the weight function
\begin{align}
\label{rho}
\rho ({\bf x})=(1+|{\bf x}|^2)^{\frac{1}{2}}\,.
\end{align}
For $\lambda \in {\mathbb R}$, we consider the weighted space $L^2(\rho^\lambda ;{\mathbb R}^3)$ given by
\begin{align}
\label{Lp-weight}
&f\in {L^{2}(\rho^\lambda ;{\mathbb R}^3)} \Longleftrightarrow {\rho }^\lambda f\in L^{2}({\mathbb R}^3),
\end{align}
which is a Hilbert space when it is endowed with the inner product and the associated norm,
\begin{align}
\label{Lp-weight2}
&(f,g)_{L^{2}(\rho^\lambda ;{\mathbb R}^3)} :=\int_{\mathbb R^3}fg{\rho }^{2\lambda}dx,\ \|f\|^2_{L^{2}(\rho^\lambda ;{\mathbb R}^3)}:= (f,f)_{L^{2}(\rho^\lambda ;{\mathbb R}^3)}.
\end{align}
We also consider the weighted Sobolev space
\begin{align}
\label{weight-1}
{\mathcal H}^{1}({\mathbb R}^3):&=\left\{f\in {\mathcal D}'({\mathbb R}^3):{\rho }^{-1}f\in L^2({\mathbb R}^3),\ \nabla f\in L^2({\mathbb R}^3)^3\right\}\nonumber\\
&=\left\{f\in L^2(\rho ^{-1};{\mathbb R}^3):\nabla f\in L^2({\mathbb R}^3)^3\right\},
\end{align}
which is a Hilbert space with respect to the inner product
\begin{equation}
\label{weight-2ip}
(f,g)_{{\mathcal H}^{1}({\mathbb R}^3)}:=(f,g)_{L^2(\rho^{-1};{\mathbb R}^3)}+(\nabla f,\nabla g)_{L^2({\mathbb R}^3)^3}
\end{equation}
and the associated norm
\begin{equation}
\label{weight-2}
\|f\|_{{\mathcal H}^{1}({\mathbb R}^3)}^2:=\left\|{\rho }^{-1}f\right\|_{L^2({\mathbb R}^3)}^2+\|\nabla f\|_{L^2({\mathbb R}^3)^3}^2
\end{equation}
(cf. \cite{Ha}; see also \cite{Amrouche-1}). The spaces $L^2(\rho^\lambda ;{\Omega }_{-})$ and ${\mathcal H}^{1}(\Omega _{-})$ can be similarly defined, and ${\mathcal D}(\overline{\Omega }_{-})$ is dense in ${\mathcal H}^{1}(\Omega _{-})$ (see, e.g.,  \cite[Theorem I.1]{Ha}, \cite[Ch.1, Theorem 2.1]{Giroire1987}).
The seminorm
\begin{align}
\label{seminorm}
|f|_{{\mathcal H}^{1}(\Omega _{-})}:=\|\nabla f\|_{L^2(\Omega _{-})^3}
\end{align}
is equivalent to the norm of ${\mathcal H}^{1}(\Omega _{-})$ defined as in \eqref{weight-2}, with $\Omega _{-}$ in place of ${\mathbb R}^3$ (see, e.g., \cite[{Chapter XI, Part B, \S 1,} Theorem 1]{Do-Li}). The {weighted} spaces $L^2(\rho^{-1};{\Omega }_{+})$ and ${\mathcal H}^{1}(\Omega _{+})$ coincide with the standard spaces $L^2({\Omega }_{+})$ and $H^{1}(\Omega _{+})$, respectively (with equivalent norms).

Note that the result in Lemma \ref{trace-operator1} extends also to the weighted Sobolev space ${\mathcal H}^{1}({\Omega }_{-})$. Therefore, there exists a linear bounded {\it exterior trace operator}
\begin{align}
\label{ext-trace}
\gamma_{-}:{\mathcal H}^{1}({\Omega }_{-})\to H^{\frac{1}{2}}({\partial\Omega }),
\end{align}
which is also surjective {(see \cite[p. 69]{Sa-Se})}. Moreover,
the embedding of the space $H^{1}(\Omega _{-})$ into ${\mathcal H}^{1}(\Omega _{-})$ and Lemma~\ref{trace-operator1} show the existence of a (non-unique) linear and bounded right inverse $\gamma^{-1}_-: H^{\frac{1}{2}}({\partial\Omega })\to \mathcal{H}^{1}({\Omega }_-)$ of operator \eqref{ext-trace}
(see \cite[Lemma 2.2]{K-L-M-W}, \cite[Theorem 2.3, Lemma 2.6]{Mikh}).

{Let $\mathring{\mathcal H}^{1}(\Omega _{-})\subset {\mathcal H}^{1}(\Omega_-)$ denote the closure of ${\mathcal D}({\Omega  }_{-})$ in ${\mathcal H}^{1}(\Omega _{-})$. This space can be characterized as
\begin{equation}
\label{property}
\mathring{\mathcal H}^{1}(\Omega _{-})=\big\{v\in {{\mathcal H}^{1}}(\Omega _{-}):\gamma_{-}v=0 \mbox{ on } \partial \Omega \big\}
\end{equation}
(cf., e.g., \cite[Theorem 3.33]{Lean}). Also let $\widetilde{\mathcal H}^{1}(\Omega _{-})\subset {\mathcal H}^{1}(\mathbb R^3)$ denote the closure of ${\mathcal D}({\Omega  }_{-})$ in ${\mathcal H}^{1}(\mathbb R^3)$. This space can be also characterized as
\begin{align}
\widetilde{\mathcal H}^{1}(\Omega _{-})=\{u\in {\mathcal H}^{1}(\mathbb R^3):{\rm{supp}}\, {u}\subseteq \overline{\Omega }_{-}\},
\end{align}
and can be isomorphically identified with the space $\mathring{\mathcal H}^{1}(\Omega _{-})$ via the extension by zero operator $\mathring E_{-}$, i.e., $\widetilde{\mathcal H}^{1}(\Omega _{-})=\mathring E_-\mathring{\mathcal H}^{1}(\Omega _{-})$ (cf., e.g., \cite[Theorem 3.29 (ii)]{Lean}).
In addition, consider the spaces (see, e.g.,  \cite[p. 44]{Amrouche-1}, \cite[Theorem 2.4]{Lang-Mendez})
\begin{align*}
{\mathcal H}^{-1}(\mathbb R^3)\!:=\!\big({\mathcal H}^{1}(\mathbb R^3)\big)^*, {\mathcal H}^{-1}(\Omega _{-}):=\big(\widetilde{\mathcal H}^{1}(\Omega _{-})\big)^*,
\widetilde {\mathcal H}^{-1}(\Omega _{-})\!:=\!\left({\mathcal H}^{1}(\Omega _{-})\right)^*.
\end{align*}

\section{The conormal derivative operators for the Stokes system with $L^{\infty}$ coefficients}

In the sequel we assume that the viscosity coefficient $\mu$ of the Stokes system \eqref{Stokes} belongs to $L^{\infty }({\mathbb R}^3)$ and there exists a constant $c_\mu>0$, such that
\begin{align}
\label{mu}
c_\mu^{-1}\leq\mu \leq {c_\mu} \mbox{ a.e. in } {\mathbb R}^3.
\end{align}

Let ${\mathbb E}({\bf u}):=\frac{1}{2}(\nabla{\bf u}+(\nabla{\bf u})^\top)$ be the strain rate tensor.
If $({\bf u},\pi)\!\in \!C^1(\overline \Omega_{\pm})^3\!\times \!C^0(\overline\Omega_{\pm})$, we can define the {\em classical} interior and exterior conormal derivatives (i.e., {\it the boundary traction fields}) for the Stokes system \eqref{Stokes} with continuously differentiable viscosity coefficient $\mu$ by the well-known formula
\begin{align}
\label{2.37-}
{\bf t}_{\mu}^{c\pm}({\bf u},\pi ):={\gamma_\pm}\left(-\pi{\mathbb I}+2{\mu}{\mathbb E}({\bf u})\right)\boldsymbol\nu ,
\end{align}
where $\boldsymbol\nu $ is the outward unit normal to $\Omega_{ +}$, defined a.e. on $\partial {\Omega }$.
Then for any function $\boldsymbol\varphi \in {\mathcal D}({\mathbb R}^3)^3$ we obtain the first Green identity
\begin{align}
{\pm}\left\langle {\bf t}_{\mu}^{c\pm}({\bf u},\pi ),
\boldsymbol\varphi \right\rangle _{_{\!\partial\Omega  }}
= &2\langle {\mu}{\mathbb E}({\bf u}),{\mathbb E}(\boldsymbol\varphi )\rangle _{\Omega_{\pm}}
-\langle \pi,{\rm{div}}\, \boldsymbol\varphi \rangle _{\Omega_{\pm}}
+\left\langle {\mathcal L}_{\mu }({\bf u},\pi ),\boldsymbol\varphi\right\rangle _{{\Omega_{\pm}}}.\nonumber
\end{align}
This formula suggests the following weak definition of the generalized conormal derivative for the Stokes system with $L^{\infty}$ coefficients in the setting of $L^2$-weighted Sobolev spaces (cf., e.g., \cite[Lemma 3.2]{Co}, \cite[Lemma 2.9]{K-L-M-W}, \cite[Lemma 2.2]{K-M-W4}, \cite[Definition 3.1, Theorem 3.2]{Mikh}, \cite[Theorem 10.4.1]{M-W}).
\begin{defn}
\label{conormal-derivative-var-Brinkman}
Let $\mu \in L^{\infty }({\mathbb R}^3)$ satisfy assumption \eqref{mu}. Let
\begin{multline}
\label{conormal-derivative-var-Brinkman-1}
{\pmb{\mathcal H}}^{1}({\Omega_{\pm}},{\mathcal L}_{\mu }):=\Big\{({\bf u}_{\pm},\pi_{\pm},{\tilde{\bf f}}_{\pm})\in {\mathcal H}^{1}({\Omega_{\pm}})^3\times L^2({\Omega_{\pm}})\times \widetilde{\mathcal H}^{-1}({\Omega_{\pm}})^3:\\
{\boldsymbol{\mathcal L}}_{\mu }({\bf u}_{\pm},\pi_{\pm} )={\tilde{\bf f}}_{\pm}|_{\Omega_{\pm}} \mbox{ and } {\rm{div}}\ {\bf u}_{\pm}=0 \mbox{ in } {\Omega_{\pm}}\Big\}.
\end{multline}
Then define the conormal derivative operator ${\bf t}_{\mu }^{\pm}\!:\!{\pmb{\mathcal H}}^{1}({\Omega_{\pm}},{\boldsymbol{\mathcal L}}_{\mu })\!\to \!H^{-\frac{1}{2}}(\partial\Omega )^3$,
\begin{align}
\label{conormal-derivative-var-Brinkman-2}
&{\pmb{\mathcal H}}^{1}({\Omega_{\pm}},{\boldsymbol{\mathcal L}}_{\mu })\ni ({\bf u}_{\pm},\pi_{\pm} ,{\tilde{\bf f}}_{\pm})\longmapsto {\bf t}_{{\mu }}^{\pm}({\bf u}_{\pm},\pi_{\pm} ;{\tilde{\bf f}}_{\pm})\in H^{-\frac{1}{2}}(\partial\Omega )^3,\\
\label{conormal-derivative-var-Brinkman-3}
&{\pm}\left\langle {\bf t}_{{\mu }}^{\pm}({\bf u}_{\pm},\pi_{\pm} ;{\tilde{\bf f}}_{\pm}),{\Phi }\right\rangle _{_{\!\partial\Omega  }}:=2\langle \mu {\mathbb E}({\bf u}_{\pm}),{\mathbb E}(\gamma^{-1}_{\pm}{\Phi})\rangle _{\Omega_{\pm}}\nonumber\\
&\hspace{2em}-\langle {\pi_{\pm}},{\rm{div}}(\gamma^{-1}_{\pm}{\Phi})\rangle _{\Omega_{\pm}}+\langle {\tilde{\bf f}}_{\pm},\gamma^{-1}_{\pm}{\Phi}\rangle _{{\Omega_{\pm}}},\ \forall \, \Phi \in H^{\frac{1}{2}}(\partial\Omega )^3,
\end{align}
where $\gamma^{-1}_{\pm}:H^{\frac{1}{2}}(\partial\Omega )^3\to {\mathcal H}^{1}({\Omega_{\pm}})^3$ is a $($non-unique$)$ bounded right inverse of the trace operator $\gamma_{\pm}:
{\mathcal H}^{1}({\Omega_{\pm}})^3\to H^{\frac{1}{2}}(\partial\Omega )^3$. 
\end{defn}
We use the simplified notation ${\bf t}_{\mu}^{\pm}({\bf u}_{\pm},\pi_{\pm})$ for ${\bf t}_{\mu}^{\pm}({\bf u}_{\pm},\pi_{\pm};{\bf 0})$.
The following assertion can be proved similar to \cite[Theorem 5.3]{Mikh-3}, \cite[Lemma 2.9]{K-L-M-W}.
\begin{lem}
\label{lem-add1}
Let $\mu \!\in \!L^{\infty }({\mathbb R}^3)$ satisfy assumption \eqref{mu}. Then for all ${\bf w}_\pm \!\in \!{\mathcal H}^{1}({\Omega_{\pm}})^3$ and $({\bf u}_{\pm},\pi_{\pm},{\tilde{\bf f}}_{\pm})\in {\pmb{\mathcal H}}^{1}({\Omega_{\pm}},{\boldsymbol{\mathcal L}}_{\mu })$
the following identity holds
\begin{align}
\label{Green-particular}
{\pm}\big\langle {\bf t}_{\mu }^{\pm}({\bf u}_{\pm},\pi_{\pm};{\tilde{\bf f}}_{\pm}),\gamma_{\pm}{\bf w}_\pm\big\rangle _{_{\partial\Omega  }}=2\langle \mu {\mathbb E}({\bf u}_{\pm}),{\mathbb E}({\bf w}_\pm)\rangle _{\Omega_{\pm}}&-\langle {\pi_{\pm}},{\rm{div}}\, {\bf w}_\pm \rangle _{\Omega_{\pm}}\nonumber\\
&+\langle {\tilde{\bf f}}_{\pm},{\bf w}_\pm \rangle _{{\Omega_{\pm}}}\,.
\end{align}
\end{lem}
\noindent Let $\gamma $ denote the trace operator from ${\mathcal H}^1({\mathbb R}^3)^3$ to $H^{\frac{1}{2}}(\partial \Omega )^3$ $($cf., e.g., \cite[Theorem 2.3, Lemma 2.6]{Mikh}, \cite[(2.2)]{B-H}$)$.
For
$({\bf u}_{\pm},\pi_{\pm},{\tilde{\bf f}}_{\pm})\in {\pmb{\mathcal H}}^{1}({\Omega_{\pm}},{\boldsymbol{\mathcal L}}_{\mu })$, let
\begin{align}
\label{u-pi-f}
&{\bf u}:=\mathring E_+{\bf u}_+ + \mathring E_-{\bf u}_-,\
\pi:=\mathring E_+\pi_+ + \mathring E_-\pi_-,\
{\bf f}:={\tilde{\bf f}}_+ + {\tilde{\bf f}}_-\\
\label{jt}
&[{\bf t}_{\mu }({\bf u},\pi;{\bf f})]:=
\!{\bf t}_{\mu }^{+}({\bf u}_+,\pi_+;\tilde{\bf f}_+)\!-\!{\bf t}_{\mu }^{-}({\bf u}_-,\pi_-;\tilde{\bf f}_-).
\end{align}
Moreover, if ${\bf f}={\bf 0}$, we define
\begin{align}
\label{jt0}
&[{\bf t}_{\mu }({\bf u},\pi)]:=[{\bf t}_{\mu }({\bf u},\pi;{\bf 0})]
=\!{\bf t}_{\mu }^{+}({\bf u}_+,\pi_+)\!-\!{\bf t}_{\mu }^{-}({\bf u}_-,\pi_-).
\end{align}
Then Lemma \ref{lem-add1} leads to the following result.
\begin{lem}
\label{lemma-add-new-1}
Let $\mu \!\in \!L^{\infty }({\mathbb R}^3)$ satisfy assumption \eqref{mu}. Also let
$({\bf u}_{\pm},\pi_{\pm},{\tilde{\bf f}}_{\pm})\in {\pmb{\mathcal H}}^{1}({\Omega_{\pm}},{\boldsymbol{\mathcal L}}_{\mu })$ and let $({\bf u},\pi ,{\bf f})$ be defined as in \eqref{u-pi-f}.
Then for all ${\bf w}\!\in \!{\mathcal H}^{1}(\mathbb R^3)^3$, the following formula holds
\begin{align}
\label{Green-particular}
\big\langle [{\bf t}_{\mu }({\bf u},\pi;{\bf f})],\gamma{\bf w}\big\rangle _{_{\partial\Omega  }}
=&2\langle \mu {\mathbb E}({\bf u}_+),{\mathbb E}({\bf w})\rangle _{\Omega_+}
+2\langle \mu {\mathbb E}({\bf u}_-),{\mathbb E}({\bf w})\rangle _{\Omega_-}\nonumber\\
&-\langle {\pi},{\rm{div}}\, {\bf w} \rangle_{\mathbb R^3}
+\langle {{\bf f}},{\bf w} \rangle_{\mathbb R^3}\,.
\end{align}
\end{lem}
\noindent We also need the following particular case of Lemma \ref{lemma-add-new-1} when ${\bf f}={\bf 0}$.
\begin{lem}
\label{lemma-add-new}
Let $\mu \!\in \!L^{\infty }({\mathbb R}^3)$ satisfy assumption \eqref{mu}. Also let
$({\bf u}_{\pm},\pi_{\pm},{\bf 0})\in {\pmb{\mathcal H}}^{1}({\Omega_{\pm}},{\boldsymbol{\mathcal L}}_{\mu })$. 
Then for all ${\bf w}\!\in \!{\mathcal H}^{1}(\mathbb R^3)^3$,
\begin{align}
\label{jump-conormal-derivative-1}
\big\langle [{\bf t}_{\mu }({\bf u},\pi )],\gamma{\bf w}\big\rangle _{_{\partial\Omega  }}
=&2\langle \mu {\mathbb E}({\bf u}_+),{\mathbb E}({\bf w})\rangle _{\Omega_+}
+2\langle \mu {\mathbb E}({\bf u}_-),{\mathbb E}({\bf w})\rangle _{\Omega_-}\nonumber\\
&-\langle {\pi},{\rm{div}}\, {\bf w} \rangle_{\mathbb R^3}
\,.
\end{align}
\end{lem}

\section{Newtonian and single layer potentials for the Stokes system with $L^{\infty}$ coefficients}

Recall that the function $\mu \!\in \!L^{\infty }({\mathbb R}^3)$ satisfies conditions \eqref{mu}. Next, we define the Newtonian and single layer potentials for the $L^{\infty}$ coefficient Stokes system \eqref{Stokes}.
\subsection{Variational solution of the variable-coefficient Stokes system in ${\mathbb R}^3$.}
First we show the following useful well-posedness result.
\begin{lem}
\label{lemma-a47-1-Stokes}
Let $a_{\mu }(\cdot ,\cdot ):{\mathcal H}^1({\mathbb R}^3)^3\times {\mathcal H}^1({\mathbb R}^3)^3\to {\mathbb R}$ and $b(\cdot ,\cdot):{\mathcal H}^1({\mathbb R}^3)^3\times L^2({\mathbb R}^3)\to {\mathbb R}$ be the bilinear forms given by
\begin{align}
\label{a-v}
a_{\mu }({\bf u},{\bf v})&:=2\langle \mu {\mathbb E}({\bf u}),{\mathbb E}({\bf v})\rangle _{ {\mathbb R}^3},\ \forall \, {\bf u},{\bf v}\in {\mathcal H}^1({\mathbb R}^3)^3,\\
\label{b-v}
b({\bf v},q)&:=-\langle {\rm{div}}\, {\bf v},q\rangle _{{\mathbb R}^3},\ \forall \, {\bf v}\in {\mathcal H}^{1}({\mathbb R}^3)^3,\ \forall \, q\in L^2({\mathbb R}^3).
\end{align}
Also let $\boldsymbol\ell :{\mathcal H}^{1}({\mathbb R}^3)^3\to {\mathbb R}$ be a linear and bounded map. Then the mixed variational formulation
\begin{align}
\label{transmission-S-variational-dl-3-equiv-0}
\left\{\begin{array}{ll}
a_{\mu }({\bf u},{\bf v})+b({\bf v},p)=\boldsymbol\ell ({\bf v}),\ \forall \ {\bf v}\in {\mathcal H}^{1}({\mathbb R}^3)^3,\\
{b({\bf u},q)=0,\ \forall \, q\in L^2({\mathbb R}^3)}
\end{array}
\right.
\end{align}
is well-posed. Hence, \eqref{transmission-S-variational-dl-3-equiv-0} has a unique solution $({\bf u},p)\in {{\mathcal H}^{1}({\mathbb R}^3)^3}\times L^2({\mathbb R}^3)$ and there exists a constant $C=C(\mu )>0$ such that
\begin{align}
\label{estimate-1-wp-S}
\|{\bf u}\|_{{\mathcal H}^{1}({\mathbb R}^3)^3}+\|p\|_{L^2({\mathbb R}^3)}\leq
C\|\boldsymbol\ell \|_{{\mathcal H}^{-1}({\mathbb R}^3)^3}.
\end{align}
\end{lem}
\begin{proof}
By using conditions \eqref{mu} and definition \eqref{weight-2} of the norm of the weighted Sobolev space ${\mathcal H}^{1}({\mathbb R}^3)$ we obtain that
\begin{align}
\label{a-1-v-S}
|a_{\mu }({\bf u},{\bf v})|&\leq 2c_\mu\|{\mathbb E}({\bf u})\|_{L^2({\mathbb R}^3)^{3\times 3}}\|{\mathbb E}({\bf v})\|_{L^2({\mathbb R}^3)^{3\times 3}}\nonumber\\
&\leq 2c_\mu \|{\bf u}\|_{{\mathcal H}^1({\mathbb R}^3)^3}\|{\bf v}\|_{{\mathcal H}^1({\mathbb R}^3)^3},\ \forall\ {\bf u},{\bf v}\in {\mathcal H}^1({\mathbb R}^3)^3.
\end{align}
Moreover, by using the Korn type inequality for functions in $\mathcal H^1({\mathbb R}^3)^3$,
\begin{align}
\label{Korn3-R3}
\|{\rm grad}\,{\bf v}\|_{L^2({\mathbb R}^3)^{3\times 3}}\leq 2^{\frac{1}{2}}\|\mathbb E ({\bf v})\|_{L^2({\mathbb R}^3)^{3\times 3}}
\end{align}
(cf., e.g., \cite[(2.2)]{Sa-Se}) and since the seminorm
\begin{align}
\label{S}
|g|_{{\mathcal H}^{1}({\mathbb R}^3)}:=\|\nabla g\|_{L^2({\mathbb R}^3)^3}
\end{align}
is a norm in $\mathcal H^1({\mathbb R}^3)^3$ equivalent to the norm defined by \eqref{weight-2} (see, e.g., \cite[{Chapter XI, Part B, \S 1,} Theorem 1]{Do-Li}), there exists a constant $c_1$ such that
\begin{align}
\label{a-1-v2-S}
a_{\mu }({\bf u},{\bf u})\geq 2c_\mu^{-1}\|{\mathbb E}({\bf u})\|_{L^2({\mathbb R}^3)^{3\times 3}}^2&\geq c_\mu^{-1}\|\nabla {\bf u}\|_{L^2({\mathbb R}^3)^{3\times 3}}^2\nonumber\\
&\geq c_\mu^{-1}c_1\|{\bf u}\|_{{\mathcal H}^1({\mathbb R}^3)^3}^2,\, \forall \, {\bf u}\in {\mathcal H}^1({\mathbb R}^3)^3.
\end{align}
{Inequalities} \eqref{a-1-v-S} and \eqref{a-1-v2-S} show that $a_{\mu }(\cdot ,\cdot ):{\mathcal H}^1({\mathbb R}^3)^3\times {\mathcal H}^1({\mathbb R}^3)^3\to {\mathbb R}$ is a bounded and coercive bilinear form. Moreover, since the divergence operator
\begin{align}
\label{div-S}
{\rm{div}}:{\mathcal H}^1({\mathbb R}^3)^3\to L^2({\mathbb R}^3)
\end{align}
is bounded, then the bilinear form $b(\cdot ,\cdot):{\mathcal H}^1({\mathbb R}^3)^3\times L^2({\mathbb R}^3)\to {\mathbb R}$ is bounded as well. In addition, the operator in \eqref{div-S} is surjective (cf. \cite[Proposition 2.1]{Al-Am-1}, \cite[Proposition 2.4]{Sa-Se}) and also 
\begin{align}
{\mathcal H}^1_{\rm div}(\mathbb R^3)^3:=&\left\{{\bf w}\in {\mathcal H}^1({\mathbb R}^3)^3: {\rm{div}}\, {\bf w}=0\right\}\nonumber\\
=&\left\{{\bf w}\in {\mathcal H}^1({\mathbb R}^3)^3: b({\bf w},q)\!=\!0,\ \forall \, q\in L^2({\mathbb R}^3) \right\}.\nonumber
\end{align}
In addition, the operator in \eqref{div-S} is surjective (cf. \cite[Proposition 2.1]{Al-Am-1}, \cite[Proposition 2.4]{Sa-Se}), and hence the operator $$-{\rm{div}}:{\mathcal H}^1({\mathbb R}^3)^3/{\mathcal H}^1_{\rm div}(\mathbb R^3)^3\to L^2(\mathbb R^3)$$ is an isomorphism. Then by Lemma \ref{surj-inj-inf-sup}(ii) the bounded bilinear form $b(\cdot ,\cdot):{\mathcal H}^1({\mathbb R}^3)^3\times L^2({\mathbb R}^3)\to {\mathbb R}$ satisfies the inf-sup condition \eqref{inf-sup-sm}. Hence, there exists $\beta _0\in (0,\infty )$ such that
\begin{align}
\label{inf-sup}
{\inf _{q\in L^2({\mathbb R}^3)\setminus \{0\}}\sup _{{\bf w}\in {\mathcal H}^1({\mathbb R}^3)^3\setminus \{\bf 0\}}\frac{b({\bf w},q)}{\|{\bf w}\|_{{\mathcal H}^1({\mathbb R}^3)^3}\|q\|_{L^2({\mathbb R}^3)}}\geq \beta _0.}
\end{align}
By applying Theorem \ref{B-B}, with $X={\mathcal H}^1({\mathbb R}^3)^3$, $M=L^2({\mathbb R}^3)$, $V={\mathcal H}^1_{\rm{div}}(\mathbb R^3)^3$, we conclude that the mixed variational formulation \eqref{transmission-S-variational-dl-3-equiv-0} has a unique solution $({\bf u},p)\in {{\mathcal H}^{1}({\mathbb R}^3)^3}\times L^2({\mathbb R}^3)$ and there exists a constant $C=C(\mu )>0$ such that $({\bf u},p)$ satisfies inequality \eqref{estimate-1-wp-S}.
\end{proof}
Next we use the result of Lemma \ref{lemma-a47-1-Stokes} in order to show the well-posedness of the $L^{\infty}$ coefficient Stokes system in the space ${\mathcal H}^{1}({\mathbb R}^3)^3\times L^2({\mathbb R}^3)$ (see also \cite[Theorem 3]{Al-Am-1} for the constant-coefficient case).
\begin{thm}
\label{Stokes-problem}
Let $\mu \!\in \!L^{\infty }({\mathbb R}^3)$ satisfy conditions \eqref{mu}.
Then the $L^{\infty}$ coefficient Stokes system
\begin{eqnarray}
\label{Newtonian-S}
\left\{
\begin{array}{ll}
{\rm{div}}\left(2\mu {\mathbb E}({\bf u})\right)-\nabla \pi =\boldsymbol\ell , & \boldsymbol\ell \in {\mathcal H}^{-1}({\mathbb R}^3)^3,
\\
{\rm{div}}\, {\bf u}=0, & \mbox{ in } {\mathbb R}^3,
\end{array}\right.
\end{eqnarray}
has a unique solution $({\bf u},p)\in {\mathcal H}^{1}({\mathbb R}^3)^3\times L^2({\mathbb R}^3)$, and there exists a constant $C_0=C_0(\mu )>0$ such that
\begin{align}
\label{estimate-1-wp-S1}
\|{\bf u}\|_{{\mathcal H}^{1}({\mathbb R}^3)^3}+ \|p\|_{L^2({\mathbb R}^3)}\leq
C_0\|\boldsymbol\ell \|_{{\mathcal H}^{-1}({\mathbb R}^3)^3}.
\end{align}
\end{thm}
\begin{proof}
Note that the Stokes system \eqref{Newtonian-S} is equivalent to the variational problem \eqref{transmission-S-variational-dl-3-equiv-0} as follows from the density of ${\mathcal D}({\mathbb R}^3)^3$ in the space ${\mathcal H}^1({\mathbb R}^3)^3$ (cf., e.g., \cite{Ha}, \cite[Proposition 2.1]{Sa-Se}). Then the well-posedness result of the Stokes system with $L^{\infty}$ coefficients \eqref{Newtonian-S} follows from Lemma \ref{lemma-a47-1-Stokes}.
\end{proof}

\subsection{\bf Newtonian potential for the Stokes system with $L^{\infty}$ coefficients}

The well-posedness of problem \eqref{Newtonian-S} allows us to define the {\it Newtonian potential for the Stokes system with $L^{\infty}$ coefficients} as follows.
\begin{defn}
\label{Newtonian-S-var-variable}
For any $\boldsymbol\ell \in {\mathcal H}^{-1}({\mathbb R}^3)^3$, we define the {\it Newtonian velocity and pressure potentials} for the Stokes system with $L^{\infty}$ coefficients
as
\begin{align}
\label{Newtonian-S-var}
\boldsymbol{\mathcal N}_{{\mu ;{\mathbb R}^3}}{\boldsymbol\ell }:={\bf u},\ {\mathcal Q}_{{\mu ;{\mathbb R}^3}}{\boldsymbol\ell }:=\pi,
\end{align}
where $({\bf u},\pi )\in {\mathcal H}^{1}({\mathbb R}^3)^3\times L^2({\mathbb R}^3)$ is the unique solution of problem \eqref{Newtonian-S} with the given datum $\boldsymbol\ell $.
\end{defn}
Moreover, the well-posedness of problem \eqref{Newtonian-S} 
yields the continuity of the above operators as stated in the following assertion (cf. also \cite[Lemma A.3]{K-L-M-W} for $\mu =1$).
\begin{lem}
\label{Newtonian-S-var-1}
The Newtonian velocity and pressure potential operators
\begin{align}
\label{Newtonian-S-var-2}
\boldsymbol{\mathcal N}_{{\mu ;{\mathbb R}^3}}:{\mathcal H}^{-1}({\mathbb R}^3)^3\to {\mathcal H}^{1}({\mathbb R}^3)^3,\
{\mathcal Q}_{{\mu ;{\mathbb R}^3}}:{\mathcal H}^{-1}({\mathbb R}^3)^3\to L^2({\mathbb R}^3)
\end{align}
are linear and continuous.
\end{lem}

\subsection{\bf Single layer potential for the Stokes system with $L^{\infty}$ coefficients}

For a given $\boldsymbol\varphi \in H^{-\frac{1}{2}}(\partial\Omega )^3$, we now consider the following transmission problem for the Stokes system with $L^{\infty}$ coefficients
\begin{eqnarray}
\label{var-Stokes-transmission-sl}
\left\{
\begin{array}{ll}
{{\rm{div}}\left(2\mu {\mathbb E}({\bf u}_{\boldsymbol\varphi })\right)}-\nabla \pi _{\boldsymbol\varphi }={\bf 0} & \mbox{ in } {\mathbb R}^3\setminus \partial \Omega ,
\\
{\rm{div}}\, {\bf u}_{\boldsymbol\varphi }=0 & \mbox{ in } {\mathbb R}^3\setminus \partial \Omega ,\
\\
\left[{\bf t}_{{\mu }}({\bf u}_{{\boldsymbol\varphi}},\pi _{{\boldsymbol\varphi}})\right]
=\boldsymbol\varphi & \mbox{ on } \partial \Omega ,
\end{array}\right.
\end{eqnarray}
and show that this problem has a unique solution $\left({\bf u}_{{\boldsymbol\varphi}},\pi _{{\boldsymbol\varphi}}\right)\!\in \!{{\mathcal H}^1({\mathbb R}^3)^3\!\times \!L^2({\mathbb R}^3)}$
(cf. also \cite[Proposition 5.1]{Sa-Se} for $\mu \!=\!1$). Note that the membership of ${\bf u}_{{\boldsymbol\varphi}}$ in ${\mathcal H}^1({\mathbb R}^3)^3$ implies the transmission condition
\begin{align}
\label{var-Stokes-transmission-sl-tu}
\left[{\gamma }({\bf u}_{{\boldsymbol\varphi}})\right]={\bf 0} \mbox{ on } \partial \Omega \,,
\end{align}
and the first equation in \eqref{var-Stokes-transmission-sl} implies also that the jump $\left[{\bf t}_{{\mu }}({\bf u}_{{\boldsymbol\varphi}},\pi _{{\boldsymbol\varphi}})\right]$ is well defined.
\begin{thm}
\label{slp-var-S-1}
Let $\boldsymbol\varphi \in H^{-\frac{1}{2}}(\partial\Omega )^3$ be given. Then the transmission problem \eqref{var-Stokes-transmission-sl} has the following equivalent mixed variational formulation: Find  $({\bf u}_{\varphi},\pi _{{\boldsymbol\varphi}})\in {\mathcal H}^1({\mathbb R}^3)^3\times L^2({\mathbb R}^3)$ such that
\begin{equation}
\label{single-layer-S-transmission}
\!\!\!\!\!\left\{\begin{array}{lll}
2\langle \mu {\mathbb E}({\bf u}_{\boldsymbol\varphi}),{\mathbb E}({\bf v})\rangle _{{{\mathbb R}^3}}\!-\!\langle \pi _{\boldsymbol\varphi},{\rm{div}}\, {\bf v}\rangle _{{{\mathbb R}^3}}
=\langle \boldsymbol\varphi ,\gamma {\bf v}\rangle _{\partial \Omega },\, \forall \, {\bf v}\in {\mathcal H}^{1}({\mathbb R}^3)^3,\\
\langle {\rm{div}}\, {\bf u}_{{\boldsymbol\varphi}},q\rangle _{{\mathbb R}^3}=0,\ \forall \, q\in L^2({\mathbb R}^3).
\end{array}\right.
\end{equation}
Moreover, problem \eqref{single-layer-S-transmission} is well-posed. Hence \eqref{single-layer-S-transmission} has a unique solution $({\bf u}_{{\boldsymbol\varphi}},\pi _{\boldsymbol\varphi})\!\in \!{\mathcal H}^1({\mathbb R}^3)^3\times L^2({\mathbb R}^3)$, and there exists a constant $C\!=\!C(\mu )$ such that
\begin{align}
\label{estimate-4S-var}
\|{\bf u}_{{\boldsymbol\varphi}}\|_{{\mathcal H}^{1}({\mathbb R}^3)^3}+\|\pi _{\boldsymbol\varphi}\|_{L^2({\mathbb R}^3)}\leq C\|{\boldsymbol\varphi}\|_{H^{-\frac{1}{2}}(\partial \Omega )^3}.
\end{align}
\end{thm}
\begin{proof}
The equivalence between the transmission problem \eqref{var-Stokes-transmission-sl} and the variational problem \eqref{single-layer-S-transmission} follows from the density of the space ${\mathcal D}({\mathbb R}^3)^3$ in ${\mathcal H}^1({\mathbb R}^3)^3$ and formula \eqref{jump-conormal-derivative-1}, while the well-posedness of the variational problem \eqref{single-layer-S-transmission} is an immediate consequence of Lemma \ref{lemma-a47-1-Stokes} with the linear form $\boldsymbol\ell :{\mathcal H}^{1}({\mathbb R}^3)^3\to {\mathbb R}$ given by
\begin{align*}
\boldsymbol\ell ({\bf v}):=\langle \boldsymbol\varphi ,\gamma {\bf v}\rangle _{\partial \Omega }=
\langle \gamma ^*\boldsymbol\varphi ,{\bf v}\rangle _{{\mathbb R}^3},\, \forall \, {\bf v}\in {\mathcal H}^{1}({\mathbb R}^3)^3,
\end{align*}
and hence $\boldsymbol\ell =\gamma ^*\boldsymbol\varphi $, where $\gamma ^*:H^{-\frac{1}{2}}(\partial \Omega )^3\to {\mathcal H}^{-1}({\mathbb R}^3)^3$ is the adjoint of the trace operator $\gamma :{\mathcal H}^{1}({\mathbb R}^3)^3\to H^{\frac{1}{2}}(\partial \Omega )^3$.
\end{proof}
Theorem \ref{slp-var-S-1} leads to the following definition (cf. \cite[p. 75]{Sa-Se} for $\mu=1$).
\begin{defn}
\label{s-l-S-variational-variable}
For any ${\boldsymbol\varphi}\in H^{-\frac{1}{2}}(\partial \Omega )^3$, we define the {\it single layer velocity and pressure potentials} for the Stokes system with $L^{\infty}$ coefficients \eqref{Stokes} as
\begin{align}
\label{slp-S-vp-var}
{\bf V}_{{\mu ;\partial \Omega }}{\boldsymbol\varphi}:={\bf u}_{{\boldsymbol\varphi}},\ {\mathcal Q}^s_{{\mu ;\partial \Omega }}{\boldsymbol\varphi}:=\pi _{{\boldsymbol\varphi}},
\end{align}
and the {\it potential operators} ${\boldsymbol{\mathcal V}}_{{\mu ;\partial \Omega }}:{\mathcal V}_{{\mu ;\partial \Omega }}:H^{-\frac{1}{2}}(\partial\Omega )^3\to H^{\frac{1}{2}}(\partial\Omega )^3$ and ${\bf K}_{\mu ;\partial\Omega }^*:{\mathcal V}_{{\mu ;\partial \Omega }}:H^{-\frac{1}{2}}(\partial\Omega )^3\to H^{-\frac{1}{2}}(\partial\Omega )^3$ as
\begin{align}
\label{slp-S-oper-var}
{\boldsymbol{\mathcal V}}_{{{\mu ;\partial \Omega }}}{\boldsymbol\varphi}:=\gamma{\bf u}_{\boldsymbol\varphi},\
{\bf K}_{\mu ;\partial\Omega }^*\boldsymbol\varphi:=
\frac{1}{2}\left(
{\bf t}_{\mu }^+({\bf u}_{\boldsymbol\varphi},\pi_{\boldsymbol\varphi})
+{\bf t}_{\mu }^-({\bf u}_{\boldsymbol\varphi},\pi_{\boldsymbol\varphi})\right),
\end{align}
where $({\bf u}_{{\boldsymbol\varphi}},\pi _{{\boldsymbol\varphi}})$ is the unique solution of problem \eqref{var-Stokes-transmission-sl} in ${{\mathcal H}^{1}({\mathbb R}^3)^3\times L^2({\mathbb R}^3)}$.
\end{defn}
The next result shows the continuity of single layer velocity and pressure operators for the variable coefficient Stokes system (cf. \cite[Proposition 5.2]{Sa-Se}, \cite[Lemma A.4, (A.10), (A.12)]{K-L-M-W} and \cite[Theorem 10.5.3]{M-W} in the case $\mu =1$).
\begin{lem}
\label{continuity-sl-S-h-var}
The following operators are linear and continuous
\begin{align}
\label{s-l-S-v1-var}
&{\bf V}_{{\mu ;\partial \Omega }}:H^{-\frac{1}{2}}(\partial\Omega )^3\to {\mathcal H}^{1}({\mathbb R}^3)^3,\
{\mathcal Q}_{{\mu ;\partial \Omega }}^s:H^{-\frac{1}{2}}(\partial\Omega )^3\to L^2({\mathbb R}^3),
\\
\label{s-l-S-v2-var}
&\boldsymbol{\mathcal V}_{{\mu ;\partial \Omega }}:H^{-\frac{1}{2}}(\partial\Omega )^3\to H^{\frac{1}{2}}(\partial\Omega )^3,\, {\bf K}_{{\mu ;\partial \Omega }}^*:H^{-\frac{1}{2}}(\partial\Omega )^3\to H^{-\frac{1}{2}}(\partial\Omega )^3.
\end{align}
\end{lem}
\begin{proof}
The continuity of operators \eqref{s-l-S-v1-var} and \eqref{s-l-S-v2-var} follows from the well-posedness of the transmission problem \eqref{var-Stokes-transmission-sl} and Definition \ref{s-l-S-variational-variable}. 
\end{proof}
The next result yields the jump relations of the single layer potential and its conormal derivative across $\partial \Omega $ (see also \cite[Proposition 5.3]{Sa-Se} for $\mu =1$).
\begin{lem}
\label{jump-s-l}
Let ${\boldsymbol\varphi}\!\in \!H^{-\frac{1}{2}}(\partial \Omega )^3$. Then almost everywhere on $\partial \Omega $,
\begin{align}
\label{sl-identities-var}
&\!\!\!\!\!\!\left[\gamma {\bf V}_{\mu ;\partial\Omega }{\boldsymbol\varphi}\right]={\bf 0},\\
\label{sl-identities-var-1}
&\!\!\!\!\!\!\left[{\bf t}_{\mu }\left({\bf V}_{\mu ;\partial\Omega }{\boldsymbol\varphi},{\mathcal Q}_{\mu ;\partial\Omega }^s{\boldsymbol\varphi}\right)\right]\!=\!\boldsymbol\varphi ,\, {\bf t}_{\mu }^{\pm }\left({\bf V}_{\mu ;\partial\Omega }{\boldsymbol\varphi},{\mathcal Q}_{\mu ;\partial\Omega }^s{\boldsymbol\varphi}\right)\!=\!\pm \frac{1}{2}\boldsymbol\varphi \!+\!{\bf K}_{\mu ;\partial\Omega }^*\boldsymbol\varphi .
\end{align}
\end{lem}
\begin{proof}
Formulas \eqref{sl-identities-var} and \eqref{sl-identities-var-1} follow from Definition \ref{s-l-S-variational-variable} and the transmission condition in \eqref{var-Stokes-transmission-sl-tu}, as well as the transmission condition in the third line of \eqref{var-Stokes-transmission-sl}.
\end{proof}
Let ${\mathbb R}\boldsymbol \nu =\{c\boldsymbol \nu :c\in {\mathbb R}\}$. Let ${\rm{Ker}}\left\{T:X\to Y\right\}:=\left\{x\in X:T(x)=0\right\}$ denote the null space of the map $T:X\to Y$.
We next obtain the main properties of the single layer potential operator (cf., e.g., \cite[Theorem 10.5.3]{M-W}, and \cite[Proposition 3.3(c)]{B-H} and \cite[Proposition 5.4]{Sa-Se} for $\mu =1$ and $\alpha \in [0,\infty )$).
\begin{lem}
The following properties hold
\begin{align}
\label{sl-n}
&{\bf V}_{\mu ;\partial \Omega }{\boldsymbol\nu }={\bf 0} \mbox{ in } {\mathbb R}^3,\,  {\mathcal Q}^s_{\mu ;\partial \Omega }{\boldsymbol\nu }=-\chi _{\Omega _{+}}\\
\label{kernel-sl-v}
&{{\rm{Ker}}\left\{\boldsymbol{\mathcal V}_{\mu ;\partial\Omega }:H^{-\frac{1}{2}}(\partial\Omega )^3\to H^{\frac{1}{2}}(\partial\Omega )^3\right\}={\mathbb R}\boldsymbol \nu ,}\\
\label{range-sl-v}
&\boldsymbol{\mathcal V}_{\mu ;\partial\Omega }\boldsymbol \varphi \in H_{\boldsymbol \nu }^{\frac{1}{2}}(\partial \Omega )^3,\, \forall \, \boldsymbol \varphi \in H^{-\frac{1}{2}}(\partial \Omega )^3,
\end{align}
where 
$\chi _{\Omega _{+}}\!=\!1$ in $\Omega _{+}$, $\chi _{\Omega _{+}}\!=\!0$ in $\Omega _{-}$, and
\begin{align}
\label{orth-nu}
H_{\boldsymbol \nu }^{\frac{1}{2}}(\partial \Omega )^3:=\big\{\boldsymbol \phi \in H^{\frac{1}{2}}(\partial \Omega )^3: \langle \boldsymbol \nu ,\boldsymbol \phi \rangle _{\partial \Omega }=0\big\}.
\end{align}
\end{lem}
\begin{proof}
First, we consider the transmission problem \eqref{var-Stokes-transmission-sl} with the datum $\boldsymbol \varphi \!=\!\boldsymbol \nu \!\in \!H^{-\frac{1}{2}}(\partial \Omega )^3$. Then the solution of this problem is given by
\begin{align}
\label{solution-nu}
\left({\bf u}_{\boldsymbol \nu },\pi _{\boldsymbol \nu }\right)=\left({\bf 0},-\chi _{\Omega _{+}}\right)\in {\mathcal H}^{1}({\mathbb R}^3)^3\times L^2({\mathbb R}^3).
\end{align}
Indeed, the pair $\left({\bf u}_{\boldsymbol \nu },\pi _{\boldsymbol \nu }\right)$ satisfies the equations and the transmission condition in \eqref{var-Stokes-transmission-sl}, as well as the transmission condition \eqref{var-Stokes-transmission-sl-tu}, and, in view of formula \eqref{jump-conormal-derivative-1} and the divergence theorem,
\begin{align}
\langle [{\bf t}_{\mu }({\bf u}_{\boldsymbol \nu },\pi_{\boldsymbol \nu })],\gamma {\bf v}\rangle _{\partial \Omega }=-\langle \pi _{\boldsymbol \nu },{\rm{div}}\, {\bf v}\rangle _{{\mathbb R}^3}=\langle {\boldsymbol \nu },\gamma {\bf v}\rangle _{\partial \Omega },\, \forall \, {\bf v}\in {\mathcal D}({\mathbb R}^3)^3.
\end{align}
Then by formula \eqref{spaces-Sobolev-inverse}, Lemma \ref{trace-operator1}, the dense embedding of the space ${\mathcal D}({\mathbb R}^3)^3$ in $H^1({\mathbb R}^3)^3$, and the above equality, we obtain that $\langle [{\bf t}_{\mu }({\bf u}_{\boldsymbol \nu },\pi_{\boldsymbol \nu })],\Phi \rangle _{\partial \Omega }=\langle {\boldsymbol \nu },\Phi \rangle _{\partial \Omega }$ for any $\Phi \in H^{\frac{1}{2}}(\partial \Omega )^3$. Hence, $[{\bf t}_{\mu }({\bf u}_{\boldsymbol \nu },\pi_{\boldsymbol \nu })]=\boldsymbol \nu $, as asserted.
Then Definition \ref{s-l-S-variational-variable} implies relations \eqref{sl-n}. Moreover, ${\boldsymbol{\mathcal V}}_{\mu ;\partial \Omega }{\boldsymbol\nu }={\bf 0}$, i.e.,
${\mathbb R}{\boldsymbol\nu }\subseteq {\rm{Ker}}\big\{{\boldsymbol{\mathcal V}}_{\mu ;\partial\Omega }:H^{-\frac{1}{2}}(\partial\Omega )^3\to H^{\frac{1}{2}}(\partial\Omega )^3\big\}.$

Now let $\boldsymbol \varphi _0\in {\rm{Ker}}\left\{{\boldsymbol{\mathcal V}}_{\mu ;\partial\Omega }
:H^{-\frac{1}{2}}(\partial\Omega )^3\to H^{\frac{1}{2}}(\partial\Omega )^3\right\}$. Let
$({\bf u}_{\boldsymbol \varphi _0},\pi_{\boldsymbol \varphi _0})=\big({\bf V}_{\mu ;\partial \Omega }{\boldsymbol\varphi _0},{\mathcal Q}^s_{\mu ;\partial \Omega }{\boldsymbol\varphi _0}\big)\in \!{{\mathcal H}^1({\mathbb R}^3)^3\!\times \!L^2({\mathbb R}^3)}$ be the unique solution of problem \eqref{var-Stokes-transmission-sl} with datum $\boldsymbol \varphi _0$. Since $\gamma {\bf u}_{\boldsymbol \varphi _0}={\bf 0}$ on $\partial \Omega $, formula \eqref{jump-conormal-derivative-1} yields that
\begin{align}
0=\langle [{\bf t}_{\mu }({\bf u}_{\boldsymbol \varphi _0},\pi_{\boldsymbol \varphi _0})],\gamma {\bf u}_{\boldsymbol \varphi _0}\rangle _{\partial \Omega }=a_{\mu }\left({\bf u}_{\boldsymbol \varphi _0},{\bf u}_{\boldsymbol \varphi _0}\right),
\end{align}
and hence ${\bf u}_{\boldsymbol \varphi _0}\!=\!{\bf 0},\, \pi _{\boldsymbol \varphi _0}\!=\!c\chi _{\Omega _{+}}$ in ${\mathbb R}^3$, where $c\!\in \!{\mathbb R}$. In view of formula \eqref{jump-conormal-derivative-1},
\begin{align}
\langle [{\bf t}_{\mu }({\bf u}_{\boldsymbol \varphi _0},\pi_{\boldsymbol \varphi _0})],\gamma {\bf w}\rangle _{\partial \Omega }=-\langle \pi _{\boldsymbol \varphi _0},{\rm{div}}\, {\bf w}\rangle _{{\mathbb R}^3}=-c\langle \boldsymbol \nu ,\gamma {\bf w}\rangle _{\partial \Omega },\, \forall \, {\bf w}\in {\mathcal D}({\mathbb R}^3)^3,\nonumber
\end{align}
and, thus, $\boldsymbol \varphi _0=[{\bf t}_{\mu }({\bf u}_{\boldsymbol \varphi },\pi_{\boldsymbol \varphi _0})]=-c\boldsymbol \nu $. Hence, formula \eqref{kernel-sl-v} follows.

Now let $\boldsymbol\varphi \in H^{-\frac{1}{2}}(\partial\Omega )^3$. By using the first formula in \eqref{slp-S-oper-var},
we obtain for any $\boldsymbol\varphi \in H^{-\frac{1}{2}}(\partial\Omega )^3$ that
$\left\langle {\boldsymbol{\mathcal V}}_{\mu ;\partial\Omega }\boldsymbol\varphi ,\boldsymbol\nu \right\rangle _{\partial \Omega }=
\left\langle\gamma\mathbf u_{\boldsymbol\varphi},\boldsymbol\nu \right\rangle _{\partial \Omega }
=\left\langle{\rm{div}}\,\mathbf u_{\boldsymbol\varphi}, 1 \right\rangle_{\Omega }=0,$
where ${\bf u}_{\boldsymbol\varphi}={\bf V}_{\mu ;\partial\Omega }\boldsymbol\varphi $.
Thus, we get relation \eqref{range-sl-v}.
\end{proof}

Next we use the notation {$\left[\!\left[\cdot \right]\!\right]$} for the equivalence classes of the space $H^{-\frac{1}{2}}(\partial\Omega ,\Lambda ^1TM)/{\mathbb R}\boldsymbol\nu$. Thus, any $\left[\!\left[\boldsymbol \varphi \right]\!\right]\!\in \!{H^{-\frac{1}{2}}(\partial\Omega )^3/{{\mathbb R}\boldsymbol\nu}}$ can be written as $\left[\!\left[\boldsymbol \varphi \right]\!\right]\!=\!\boldsymbol \varphi \!+\!{\mathbb R}\boldsymbol\nu $, where $\boldsymbol \varphi \in H^{-\frac{1}{2}}(\partial\Omega )^3$.

Exploiting properties \eqref{kernel-sl-v} and \eqref{range-sl-v}, we now show the following invertibility result (cf. \cite[Theorem 10.5.3]{M-W}, \cite[Proposition 3.3(d)]{B-H}, \cite[Proposition 5.5]{Sa-Se} for $\mu =1$ and $\alpha \geq 0$ constant).
\begin{lem}
\label{isom-sl-v}
The following operator is an isomorphism
\begin{align}
\label{sl-v-isom}
\boldsymbol{\mathcal V}_{\mu ;\partial\Omega }:H^{-\frac{1}{2}}(\partial\Omega )^3/{\mathbb R}\nu \to H_{\boldsymbol \nu }^{\frac{1}{2}}(\partial\Omega )^3.
\end{align}
\end{lem}
\begin{proof}
We use arguments similar to those for Proposition 5.5 in \cite{Sa-Se}. First, Lemma \ref{continuity-sl-S-h-var} and the membership relation \eqref{range-sl-v} imply that the linear operator in \eqref{sl-v-isom} is continuous.
We show that this operator is also $H^{-\frac{1}{2}}(\partial\Omega )^3\!/{\mathbb R}\!\boldsymbol\nu $-elliptic 
i.e., that there exists a constant $c=c(\partial \Omega )$ such that
\begin{align}
\label{H-elliptic-sl-v}
\left\langle \boldsymbol{\mathcal V}_{\mu ;\partial\Omega }\left[\!\left[\boldsymbol \varphi \right]\!\right],\left[\!\left[\boldsymbol \varphi \right]\!\right]\right\rangle _{\partial \Omega }\geq c\left\|\left[\!\left[\boldsymbol \varphi \right]\!\right]\right\|_{H^{-\frac{1}{2}}(\partial\Omega )^3/{{\mathbb R}\boldsymbol\nu}}^2,\ \forall \, \left[\!\left[\boldsymbol \varphi \right]\!\right]\in {H^{-\frac{1}{2}}(\partial\Omega )^3/{{\mathbb R}\boldsymbol\nu}}.
\end{align}
Let $\left[\!\left[\boldsymbol \varphi \right]\!\right]\!\in \!{H^{-\frac{1}{2}}(\partial\Omega )^3/{{\mathbb R}\boldsymbol\nu}}$. Thus, $\left[\!\left[\boldsymbol \varphi \right]\!\right]\!=\!\boldsymbol \varphi \!+\!{\mathbb R}\boldsymbol\nu $, where $\boldsymbol \varphi \in H^{-\frac{1}{2}}(\partial\Omega )^3$. In view of formula \eqref{jump-conormal-derivative-1}, Definition \ref{s-l-S-variational-variable}, relations \eqref{kernel-sl-v}, \eqref{range-sl-v}, and inequality \eqref{a-1-v2-S},
\begin{align}
\label{onto-div-2}
\!\!\!\!\!\!\left\langle \boldsymbol{\mathcal V}_{\mu ;\partial\Omega }(\left[\!\left[\boldsymbol \varphi \right]\!\right]),\left[\!\left[\boldsymbol \varphi \right]\!\right]\right\rangle _{\partial \Omega }&=\left\langle \boldsymbol{\mathcal V}_{\mu ;\partial\Omega }(\boldsymbol \varphi ),\boldsymbol \varphi \right\rangle _{\partial \Omega }\!=\!\langle \gamma {\bf u}_{\boldsymbol \varphi},[{\bf t}_{\mu }({\bf u}_{\boldsymbol \varphi},\pi_{\boldsymbol \varphi})]\rangle _{\partial \Omega }\nonumber\\
&=a_{\mu }({\bf u}_{\boldsymbol \varphi },{\bf u}_{\boldsymbol \varphi })\!\geq \!c_\mu^{-1}\|{\bf u}_{\boldsymbol \varphi }\|_{H^1({\mathbb R}^3)^3}^2,
\end{align}
where ${\bf u}_{\boldsymbol \varphi }={\bf V}_{\mu ;\partial\Omega }\boldsymbol \varphi $ and $\pi_{\boldsymbol\varphi}={\mathcal Q}^s_{{\mu ;\partial \Omega }}{\boldsymbol\varphi}$.
Now we use the property that the trace operator
\begin{align}
\label{onto-div}
\gamma :{\mathcal H}_{{\rm div}}^1({\mathbb R}^3)^3\to H_{\boldsymbol \nu }^{\frac{1}{2}}(\partial\Omega )^3
\end{align}
is surjective having a bounded right inverse $\gamma ^{-1}:H_{\boldsymbol \nu }^{\frac{1}{2}}(\partial\Omega )^3\to {\mathcal H}_{{\rm div}}^1({\mathbb R}^3)^3$ (cf., e.g., \cite[Proposition 4.4]{Sa-Se}). Hence, for any $\boldsymbol \Phi \in H_{\boldsymbol \nu }^{\frac{1}{2}}(\partial\Omega )^3$, we have that ${\bf w}=\gamma ^{-1}\boldsymbol \Phi \in {\mathcal H}_{{\rm div}}^1({\mathbb R}^3)^3.$ Then there exists $c'\equiv c'(\partial \Omega )\in (0,\infty )$ such that
\begin{align}
\label{estimate-norm}
|\langle \left[\!\left[\boldsymbol \varphi \right]\!\right],\boldsymbol \Phi \rangle _{\partial \Omega }|&=|\langle \boldsymbol \varphi ,\boldsymbol \Phi \rangle _{\partial \Omega }|=|\langle [{\bf t}_{\mu }({\bf u}_{\boldsymbol \varphi },\pi_{\boldsymbol \varphi })],\gamma {\bf w}\rangle _{\partial \Omega }|=|a_{\mu }({\bf u}_{\boldsymbol \varphi },{\bf w})|\\
&\!\leq \!2c_\mu \|{\bf u}_{\boldsymbol \varphi }\|_{{\mathcal H}^1({\mathbb R}^3)^3}\|\!\gamma ^{-1}\boldsymbol \Phi \|_{{\mathcal H}^1({\mathbb R}^3)^3}\!\!\leq \!2c_\mu c'\!\|{\bf u}_{\boldsymbol \varphi }\|_{{\mathcal H}^1({\mathbb R}^3)^3}\|\boldsymbol \Phi \|_{H^{\frac{1}{2}}(\partial\Omega )^3},\nonumber
\end{align}
where the first equality in \eqref{estimate-norm} follows from the relation $\left[\!\left[\boldsymbol \varphi \right]\!\right]=\boldsymbol \varphi +{\mathbb R}\boldsymbol \nu $ and the membership of $\boldsymbol \Phi $ in $H_{\boldsymbol \nu }^{\frac{1}{2}}(\partial\Omega )^3$, the second equality follows from Definition \ref{s-l-S-variational-variable}, 
and the third equality is a consequence of formula \eqref{jump-conormal-derivative-1}. 
Since the space $H_{\boldsymbol \nu }^{\frac{1}{2}}(\partial\Omega )^3$ is the dual of the space $H^{-\frac{1}{2}}(\partial\Omega )^3/{{\mathbb R}\nu}$, formula \eqref{estimate-norm} yields that
\begin{align}
\label{onto-div-1}
\|\left[\!\left[\boldsymbol \varphi \right]\!\right]\|_{H^{-\frac{1}{2}}(\partial\Omega )^3/{{\mathbb R}\boldsymbol \nu}}\leq 2c_\mu c'\|{\bf u}_{\boldsymbol \varphi }\|_{{\mathcal H}^{1}({\mathbb R}^3)^3}.
\end{align}
Then by \eqref{onto-div-2} and \eqref{onto-div-1} we obtain inequality \eqref{H-elliptic-sl-v}, and
the Lax-Milgram lemma yields that operator \eqref{sl-v-isom} is an isomorphism.
\end{proof}
\begin{rem}
\label{equal}
The fundamental solution of the constant-coefficient Stokes system in ${\mathbb R}^3$ is well known and leads to the construction of Newtonian and boundary layer potentials via the integral approach (see, e.g., \cite{Co,12,M-W,24}).
In view of Theorems \ref{Stokes-problem} and \ref{slp-var-S-1}, the Newtonian and single layer potentials provided by the variational approach (in the case $\mu \!=\!1$) coincide with classical ones expressed in terms of the fundamental solution, since they satisfy the same boundary value problems \eqref{Newtonian-S} and \eqref{var-Stokes-transmission-sl}, respectively (see also
\cite[Proposition 5.1]{Sa-Se} for $\mu \!=\!1$). The assumption $\mu =1$ is a particular case of a more general case of $L^{\infty}$ coefficients analyzed in this paper.
We also note that an alternative approach, reducing various boundary value problems for variable-coefficient elliptic partial differential equations to {\em boundary-domain integral equations}, by employing the explicit parametrix-based integral potentials, was explored in, e.g., \cite{CMN-1,Ch-Mi-Na,Ch-Mi-Na-3}.
\end{rem}

\section{Exterior Dirichlet problem for the Stokes system with $L^{\infty}$ coefficients}

In this section we analyze the exterior Dirichlet problem for the Stokes system with $L^{\infty}$ coefficients
\begin{equation}
\label{Dirichlet-var-Stokes}
\left\{
\begin{array}{ll}
{{\rm{div}}\left(2\mu {\mathbb E}({\bf u})\right)}-\nabla \pi ={\bf f} & \mbox{ in } \Omega _{-},
\\
{\rm{div}} \ {\bf u}=0 & \mbox{ in } \Omega _{-},\
\\
{\gamma }_{-}{\bf u}={\boldsymbol\phi} &  \mbox{ on } \partial \Omega ,
\end{array}\right.
\end{equation}
with given data $({\bf f},{\boldsymbol\phi})\in {\mathcal H}^{-1}(\Omega _{-})^3\times H^{\frac{1}{2}}(\partial \Omega )^3$.

\subsection{Variational approach}
First, we use a variational approach and show that problem \eqref{Dirichlet-var-Stokes} has a unique solution $({\bf u},\pi )\in {\mathcal H}^{1}(\Omega _{-})^3\times L^2(\Omega _{-})$ (see also \cite[Theorem 3.4]{Gi-Se} and \cite[Theorem 3.16]{Al-Am} for the constant-coefficient Stokes system).

\begin{thm}
\label{Dirichlet-ext-var-Stokes}
{Assume} that $\mu\in L^{\infty }(\Omega _{-})$ satisfies conditions \eqref{mu}.
Then for all given data $({\bf f},{\boldsymbol\phi})\in {\mathcal H}^{-1}(\Omega _{-})^3\times H^{\frac{1}{2}}(\partial \Omega )^3$ the exterior Dirichlet problem for the $L^{\infty }$ coefficient Stokes system \eqref{Dirichlet-var-Stokes}
is well posed. Hence problem \eqref{Dirichlet-var-Stokes} has a unique solution $({\bf u},\pi )\in {\mathcal H}^{1}(\Omega _{-})^3\times L^2(\Omega _{-})$ and there exists 
a constant $C\equiv C(\partial \Omega ;\mu )>0$ such that
\begin{align}
\label{estimate-Dirichlet-var-ext}
\|{\bf u}\|_{{\mathcal H}^{1}(\Omega _{-})^3}+\|\pi \|_{L^2(\Omega _{-})}\leq
C\left(\|{\bf f}\|_{{{\mathcal H}^{-1}(\Omega _{-})^3}}+\|\boldsymbol\phi \|_{H^{\frac{1}{2}}(\partial \Omega )^3}\right).
\end{align}
\end{thm}
\begin{proof}
First, we note that the density of the space ${\mathcal D}(\Omega _{-})^3$
in $\widetilde{\mathcal H}^{1}(\Omega _{-})^3$
implies that the exterior Dirichlet problem \eqref{Dirichlet-var-Stokes} has the following equivalent variational formulation: Find $({\bf u},\pi )\in {{\mathcal H}^{1}(\Omega _{-})^3\times L^2(\Omega _{-})}$ such that
\begin{equation}
\label{variational-Dirichlet-Stokes}
\left\{\begin{array}{lll}
2\langle \mu {\mathbb E}({\bf u}),{\mathbb E}(\tilde{\bf v})\rangle _{\Omega _{-}}-\langle \pi ,{\rm{div}}\ \tilde{\bf v}\rangle _{\Omega _{-}}
 =-\langle {\bf f},\tilde{\bf v}\rangle _{\Omega _{-}},\ \forall \, \tilde{\bf v}\in
\widetilde{\mathcal H}^{1}(\Omega _{-})^3,\\
\langle {\rm{div}}\ {\bf u},q\rangle _{\Omega _{-}}=0,\ \forall \, q\in {L^2(\Omega _{-})},\\
{\gamma }_{-}({\bf u})={\boldsymbol\phi} \mbox{ on } \partial \Omega .
\end{array}\right.
\end{equation}
Next, we consider ${\bf u}_0\in {\mathcal H}^{1}(\Omega _{-})^3$ such that
\begin{align}
\label{div-trace}
\left\{\begin{array}{lll}
{\rm{div}}\, {\bf u}_0=0 & \mbox{ in } \Omega _{-},\\
\gamma _{-}{\bf u}_0={\boldsymbol\phi } & \mbox{ on } \partial \Omega .
\end{array}
\right.
\end{align}
Particularly, we can choose ${\bf u}_0$ as  the solution of the Dirichlet problem for
a constant-coefficient Brinkman system
\begin{align}
\label{Br1}
\left\{\begin{array}{lll}
(\triangle -\alpha {\mathbb I}){\bf u}_0-\nabla \pi _0=0,\
{\rm{div}}\, {\bf u}_0=0 & \mbox{ in } \Omega _{-},
\\
\gamma _{-}{\bf u}_0={\boldsymbol\phi } & \mbox{ on } \partial \Omega ,
\end{array}
\right.
\end{align}
where $\alpha >0$ is an arbitrary constant. The solution is given by the corresponding double layer potential
\begin{align}
\label{div-trace-1}
{\bf u}_0={\bf W}_{\alpha ;\partial \Omega }\left(\frac{1}{2}{\mathbb I}+{\bf K}_{\alpha ;\partial \Omega }\right)^{-1}{\boldsymbol\phi},
\end{align}
where ${\bf K}_{\alpha ;\partial \Omega }:H^{\frac{1}{2}}(\partial \Omega )^3\to H^{\frac{1}{2}}(\partial \Omega )^3$ is the corresponding Brinkman double-layer boundary potential operator. Note that
\begin{align}
\label{dl-velocity}
&({\bf W}_{\alpha }{\textbf h})_{j}({\bf x}):=\int_{\partial {\Omega}}S^{\alpha}_{ij\ell }({\bf x},{\bf y})\nu _{\ell}({\bf y}) h_{i}({\bf y})d\sigma _{\bf y}.
\end{align}
The explicit form of the kernel $S^{\alpha}_{ij\ell }({\bf x},{\bf y})$
can be found in \cite[(2.14)-(2.18)]{24} and \cite[Section 3.2.1]{12}.

In addition, the operator $\frac{1}{2}{\mathbb I}+{\bf K}_{\alpha ;\partial \Omega }:H^{\frac{1}{2}}(\partial \Omega )^3\to H^{\frac{1}{2}}(\partial \Omega )^3$ is an isomorphism, and ${\bf u}_0$ belongs to the space $H^{1}(\Omega _{-})^3$ and satisfies \eqref{Br1}, and hence \eqref{div-trace}.
Moreover, the embedding $H^{1}(\Omega _{-})^3\subset {\mathcal H}^{1}(\Omega _{-})^3$ shows that ${\bf u}_0$ belongs also to the space ${\mathcal H}^{1}(\Omega _{-})^3$ (see also \cite[Lemma 3.2, Remark 3.3]{Gi-Se}).

Then with the new variable $\mathring{\bf u}:={\bf u}-{\bf u}_0\in \mathring{\mathcal H}^{1}(\Omega _{-})^3$, the variational problem \eqref{variational-Dirichlet-Stokes} reduces to the following mixed variational formulation (c.f. Problem (Q) in p. 324 of \cite{Gi-Se} for the
constant-coefficient Stokes system): Find $(\mathring{\bf u},\pi )\in \mathring{\mathcal H}^{1}(\Omega _{-})^3\times L^2(\Omega _{-})$ such that
\begin{equation}
\label{variational-Dirichlet-2-Stokes}
\left\{\begin{array}{lll}
a_{{\mu ;\Omega _{-}}}(\mathring{\bf u},{\bf v})+b_{\Omega _{-}}({\bf v},\pi )=
{\mathfrak F}_{\mu ;{\bf u}_0}({\bf v}),\ \forall \, {\bf v}\in \mathring{\mathcal H}^{1}(\Omega _{-})^3,\\
b_{_{\Omega _{-}}}(\mathring{\bf u},q)=0,\ \forall \, q\in L^2(\Omega _{-}),
\end{array}\right.
\end{equation}
where $a_{{\mu ;\Omega _{-}}}:\mathring{\mathcal H}^1(\Omega _{-})^3\times \mathring{\mathcal H}^1(\Omega _{-})^3\to {\mathbb R}$ and $b_{_{\Omega _{-}}}:\mathring{\mathcal H}^1(\Omega _{-})^3\times L^2(\Omega _{-})\to {\mathbb R}$ are the bilinear forms given by
\begin{align}
\label{a-Omega}
&a_{{\mu ;\Omega _{-}}}({\bf w},{\bf v}):=2\langle \mu {\mathbb E}({\bf w}),{\mathbb E}({\bf v})\rangle _{\Omega _{-}},\ \forall \, {\bf v},{\bf w}\in \mathring{\mathcal H}^1(\Omega _{-})^3,\\
\label{b-Omega}
&b_{_{\Omega _{-}}}({\bf v},q):=-\langle {\rm{div}}\, {\bf v},q\rangle _{\Omega _{-}},\ \forall \, {\bf v}\in \mathring{\mathcal H}^1(\Omega _{-})^3\,, q\in L^2(\Omega _{-}),
\end{align}
and ${\mathfrak F}_{\mu ;{{\bf u}_0}}:\mathring{\mathcal H}^1(\Omega _{-})^3\to {\mathbb R}$ is the linear form given by
\begin{align}
\label{l-u-0-Stokes}
{\mathfrak F}_{\mu ;{\bf u}_0}({\bf v}):=-\left(\langle {\bf f},{\mathring E}{\bf v}\rangle _{\Omega _{-}}+2\langle \mu {\mathbb E}({\bf u}_0),{\mathbb E}({\bf v})\rangle _{\Omega _{-}}\right),\ \forall \, {\bf v}\in \mathring{\mathcal H}^{1}(\Omega _{-})^3.
\end{align}
Here we took into account that the spaces $\mathring{\mathcal H}^{1}(\Omega _{-})^3$ and $\widetilde{\mathcal H}^{1}(\Omega _{-})^3$ can be identified through the isomorphism
$\mathring E_{-}:\mathring{\mathcal H}^{1}(\Omega _{-})^3\to\widetilde{\mathcal H}^{1}(\Omega _{-})^3$.
Note that
\begin{align}
\mathring{\mathcal H}_{{\rm{div}}}^1(\Omega_{-})^3:
&=\left\{{\bf v}\in \mathring{\mathcal H}^1(\Omega_{-})^3: {\rm{div}}\, {\bf v}=0 \mbox{ in } \Omega _{-}\right\}\nonumber\\
&=\left\{{\bf v}\in \mathring{\mathcal H}^1(\Omega_{-})^3: b_{_{\Omega _{-}}}({\bf v},q)=0,\ \forall \, q\in L^2(\Omega _{-})\right\}.
\end{align}
Now, formula \eqref{weight-2}, inequality \eqref{mu} and the H\"{o}lder inequality yield that
\begin{align}
\label{a-1-v-d-Stokes}
|a_{{\mu ;\Omega _{-}}}({\bf v}_1,{\bf v}_2)|&\leq 2c_\mu\|{\mathbb E}({\bf v}_1)\|_{L^2(\Omega _{-})^{3\times 3}}\|{\mathbb E}({\bf v}_2)\|_{L^2(\Omega _{-})^{3\times 3}}\nonumber\\
&\leq 2c_\mu \|{\bf v}_1\|_{{\mathcal H}^1(\Omega _{-})^3}\|{\bf v}\|_{{\mathcal H}^1(\Omega _{-})^3},\ \forall\ {\bf v}_1,\, {\bf v}_2\in \mathring{\mathcal H}^1(\Omega _{-})^3.
\end{align}
Moreover, the formula
\begin{align}
\!\!\!\!\!2\|{\mathbb E}({\bf v})\|_{L^2(\Omega _{-})^{3\times 3}}^2\!=\!\|{\rm grad}\,{\bf v}\|_{L^2(\Omega_-)^{3\times 3}}^2\!+\!\|{\rm div}\,{\bf v}\|_{L^2(\Omega_-)}^2,\, \forall \, {\bf v}\in {\mathcal D}(\Omega _{-})^3
\end{align}
(cf., e.g., the proof of Corollary 2.2 in \cite{Sa-Se}), and the density of the space ${\mathcal D}(\Omega _{-})^3$ in $\mathring{\mathcal H}^1(\Omega_-)^3$ show that the same formula holds also for any function in $\mathring{\mathcal H}^1(\Omega_-)^3$. Therefore, we obtain the following Korn type inequality
\begin{align}
\label{2nd-Korn-exerior-cal}
\|{\rm grad}\,{\bf v}\|_{L^2(\Omega_-)^{3\times 3}}\leq 2^{\frac{1}{2}}\|\mathbb E ({\bf v})\|_{L^2(\Omega_-)^{3\times 3}}\,, \ \forall \, {\bf v}\in \mathring{\mathcal H}^1(\Omega_-)^3.
\end{align}

Then by using inequality \eqref{2nd-Korn-exerior-cal}, the equivalence of seminorm \eqref{seminorm} to the norm \eqref{weight-2} in the space $\mathcal H^1(\Omega_-)^3$, and assumption \eqref{mu} we deduce that there exists a constant $C=C(\Omega _{-})>0$ such that
\begin{align*}
\|{\bf u}\|_{{\mathcal H}^1(\Omega _{-})^3}^2&\leq C\|{\rm grad}\,{\bf u}\|_{L^2(\Omega_-)^{3\times 3}}^2\!\leq 2C\|{\mathbb E}({\bf u})\|_{L^2(\Omega _{-})^{3\times 3}}^2\nonumber\\
&\leq 2C{c_\mu }\|\mu {\mathbb E}({\bf u})\|_{L^2(\Omega _{-})^{3\times 3}}^2
=2C{c_\mu }a_{{\mu ;\Omega _{-}}}({\bf u},{\bf u}),\, \forall \, {\bf u}\in \mathring{\mathcal H}^1(\Omega _{-})^3,
\end{align*}
and accordingly that 
\begin{align}
\label{a-1-v2-e-Stokes}
a_{{\mu ;\Omega _{-}}}({\bf u},{\bf u})\geq \frac{1}{2Cc_\mu }\|{\bf u}\|_{{\mathcal H}^1(\Omega _{-})^3}^2,\ \forall \, {\bf u}\in \mathring{\mathcal H}^1(\Omega _{-})^3.
\end{align}
In view of inequalities \eqref{a-1-v-d-Stokes} and \eqref{a-1-v2-e-Stokes} it follows that the bilinear form $a_{{\mu ;\Omega _{-}}}(\cdot ,\cdot ):\mathring{\mathcal H}^1(\Omega _{-})^3\times \mathring{\mathcal H}^1(\Omega _{-})^3\to {\mathbb R}$ is bounded and coercive. Moreover, arguments similar to those for inequality \eqref{a-1-v-d-Stokes} imply that the bilinear form $b_{_{\Omega _{-}}}(\cdot ,\cdot):\mathring{\mathcal H}^1(\Omega _{-})^3\times L^2(\Omega _{-})\to {\mathbb R}$ and the linear form ${\mathfrak F}_{\mu ;{\bf u}_0}:\mathring{\mathcal H}^1({\mathbb R}^3)^3\to {\mathbb R}$ given by \eqref{b-Omega} and \eqref{l-u-0-Stokes}, are also bounded.
Since the operator
\begin{align}
\label{div-tilde-Stokes}
{\rm{div}}:\mathring{\mathcal H}^1(\Omega_{-})^3\to L^2(\Omega _{-})
\end{align}
is surjective (cf., e.g., \cite[Theorem 3.2]{Gi-Se}), then by Lemma \ref{surj-inj-inf-sup}, the bounded bilinear form $b_{_{\Omega _{-}}}(\cdot ,\cdot):\mathring{\mathcal H}^1(\Omega _{-})^3\times L^2(\Omega _{-})\to {\mathbb R}$ satisfies the inf-sup condition 
\begin{align}
\label{inf-sup-Omega-Stokes}
\inf _{q\in L^2(\Omega _{-})\setminus \{0\}}\sup _{{\bf v}\in \mathring{\mathcal H}^1(\Omega _{-})^3\setminus \{\bf 0\}}\frac{b_{_{\Omega _{-}}}({\bf v},q)}{\|{\bf v}\|_{\mathring{\mathcal H}^1(\Omega _{-})^3}\|q\|_{L^2(\Omega _{-})}}\geq \beta _D
\end{align}
with some constant $\beta _D>0$ (cf. \cite[{Theorem 3.3}]{Gi-Se}).
Then Theorem \ref{B-B} (with $X=\mathring{\mathcal H}^1(\Omega _{-})^3$ and $M=L^2(\Omega _{-})$) implies that the variational problem \eqref{variational-Dirichlet-2-Stokes} has a unique solution $(\mathring{\bf u},\pi )\in \mathring{\mathcal H}^1(\Omega _{-})^3\times L^2(\Omega _{-})$. Moreover, the pair $({\bf u},\pi )=(\mathring{\bf u}+{\bf u}_0,\pi )\in {\mathcal H}^1(\Omega _{-})^3\times L^2(\Omega _{-})$, where ${\bf u}_0\in {\mathcal H}^1(\Omega _{-})^3$ satisfies relations \eqref{div-trace}, is the unique solution of the mixed variational formulation \eqref{variational-Dirichlet-Stokes} and depends continuously on the given data $({\bf f},{\boldsymbol\phi})\in {\mathcal H}^{-1}(\Omega _{-})^3\times H^{\frac{1}{2}}(\partial \Omega )^3$. The equivalence between the variational problem \eqref{variational-Dirichlet-Stokes} and the exterior Dirichlet problem \eqref{Dirichlet-var-Stokes} shows that problem \eqref{Dirichlet-var-Stokes} is also well-posed, as asserted.
\end{proof}

\subsection{Potential approach}
Theorem \ref{Dirichlet-ext-var-Stokes} asserts the well-posedness of the exterior Dirichlet problem for the Stokes system with $L^{\infty}$ coefficients. However, if the given data $({\bf f},{\boldsymbol\phi})$ belong to the space ${\mathcal H}^{-1}(\Omega _{-})^3\times H_{\boldsymbol \nu }^{\frac{1}{2}}(\partial \Omega )^3$, then the solution can be
expressed in terms of the Newtonian and single layer potential and of the inverse of the single layer operator as follows (cf. \cite[Theorem 3.4]{Gi-Se} for $\mu>0$ constant, \cite[Theorem 10.1]{Fa-Me-Mi} and \cite[Theorem 5.1]{Lang-Mendez} for the Laplace operator).
\begin{thm}
\label{Dirichlet-ext-var-Stokes-new}
If ${\bf f}\in {\mathcal H}^{-1}(\Omega _{-})^3$ and ${\boldsymbol\phi}\in H_{\boldsymbol \nu }^{\frac{1}{2}}(\partial \Omega )^3$ then the exterior Dirichlet problem \eqref{Dirichlet-var-Stokes}
has a unique solution $({\bf u},\pi )\in {\mathcal H}^{1}(\Omega _{-})^3\times L^2(\Omega _{-})$, given by
\begin{align}
\label{solution-Dirichlet-ext-0}
&{\bf u}=\boldsymbol{\mathcal N}_{{\mu ;{\mathbb R}^3}}(\tilde{\bf f})|_{\Omega _{-}}+{\bf V}_{\mu ;\partial \Omega }\left(\boldsymbol{\mathcal V}_{\mu ;\partial \Omega }^{-1}\big({\boldsymbol\phi}-\gamma _{-}\big(\boldsymbol{\mathcal N}_{{\mu ;{\mathbb R}^3}}(\tilde{\bf f})\big)\big)\right),\\
\label{solution-Dirichlet-ext-p-0}
&\pi ={\mathcal Q}_{{\mu ;{\mathbb R}^3}}(\tilde{\bf f})|_{\Omega _{-}}+{\mathcal Q}_{\mu ;\partial \Omega }^s\left(\boldsymbol{\mathcal V}_{\mu ;\partial \Omega }^{-1}\big({\boldsymbol\phi}-\gamma _{-}\big(\boldsymbol{\mathcal N}_{{\mu ;{\mathbb R}^3}}(\tilde{\bf f})\big)\big)\right) \mbox{ in } \Omega _{-},
\end{align}
where $\tilde{\bf f}$ is an extension of ${\bf f}$ to an element of ${\mathcal H}^{1}({\mathbb R}^3)^3$.
\end{thm}
\begin{proof}
The result follows from Definition \ref{Newtonian-S-var-variable} and Lemmas \ref{continuity-sl-S-h-var}, \ref{jump-s-l}, \ref{isom-sl-v}.
\end{proof}

\appendix
{
\section{Mixed variational formulations and their well-posedness property}
\label{B-B-theory}

Here we make a brief review of well-posedness results due to Babu\u{s}ka \cite{Babuska} and Brezzi \cite{Brezzi} for mixed variational formulations related to bounded bilinear forms in reflexive Banach spaces. We follow \cite[Section 2.4]{Ern-Gu}, \cite{Brezzi-Fortin}, \cite[\S 4]{Gi-Ra}.

Let $X$ and ${\mathcal M}$ be reflexive Banach spaces, and let $X^*$ and ${\mathcal M}^*$ be their dual spaces. Let $a(\cdot ,\cdot):X\!\times \!X\!\to \!{\mathbb R}$, $b(\cdot ,\cdot):X\!\times \!{\mathcal M}\!\to \! {\mathbb R}$ be {bounded bilinear forms}. 
Then we consider the following abstract mixed variational formulation.

{\it For $f\in X^{*}$, $g\in {\mathcal M}^{*}$ given, find a pair $(u,p)\in X\times {\mathcal M}$ such that}
\begin{equation}
\label{mixed-variational}
\left\{\begin{array}{ll}
a(u,v)+b(v,p)&=f(v), \ \ \forall \ v\in X,\\
b(u,q)&=g(q), \ \ \forall \ q\in {\mathcal M}.
\end{array}
\right.
\end{equation}
Let $A:X\to X^{*}$ be the bounded linear operator defined by
\begin{align}
\label{A}
\langle Av,w\rangle =a(v,w),\, \forall \, v,w\in X,
\end{align}
where $\langle \cdot ,\cdot \rangle :=\!_{X^{*}}\langle \cdot ,\cdot \rangle _X$ is the duality pairing of the dual spaces $X^{*}$ and $X$. We also use the notation $\langle \cdot ,\cdot \rangle $ for the duality pairing $_{{\mathcal M}^{*}}\langle \cdot ,\cdot \rangle _{\mathcal M}$. Let $B:X\to {\mathcal M}^{*}$ and $B^{*}:{\mathcal M}\to X^{*}$ be the bounded linear and transpose operators given by
\begin{align}
\label{B}
&\langle Bv,q\rangle =b(v,q),\ \langle v,B^*q\rangle =\langle Bv,q\rangle ,\, \forall \, v\in X,\, \forall \, q\in {\mathcal M}.
\end{align}
In addition, we consider the spaces
\begin{align}
\label{V}
&V:={\rm{Ker}}\, B=\left\{v\in X: b(v,q)=0,\ \forall \, q\in {\mathcal M}\right\},\\
&V^{\perp}:=\left\{T\in X^{*}: \langle T,v\rangle =0,\ \forall \, v\in V\right\}.
\end{align}
Then the following well-posedness result holds (cf., e.g., \cite[Theorem 2.34]{Ern-Gu}).
\begin{thm}
\label{B-N-B-new}
Let $X$ and ${\mathcal M}$ be reflexive Banach spaces, $f\in X^{*}$ and $g\in {\mathcal M}^{*}$, and $a(\cdot ,\cdot):X\times X\to {\mathbb R}$ and $b(\cdot ,\cdot):X\times {\mathcal M}\to {\mathbb R}$ be bounded bilinear forms.
Let $V$ be the subspace of $X$ defined by \eqref{V}.
Then the variational problem \eqref{mixed-variational} is well-posed if and only if
$a(\cdot ,\cdot )$ satisfies the conditions
\begin{equation}
\label{B-N-B1}
\left\{\begin{array}{lll}
\exists \ \lambda >0\ \mbox{ such that } \ \displaystyle\inf _{u\in V\setminus \{0\}}\sup_{v\in V\setminus \{0\}}\frac{a(u,v)}{\|u\|_X\|v\|_X}\geq \lambda ,\\
\{v\in V: a(u,v)=0,\ \forall \ u\in V\}=\{0\},
\end{array}\right.
\end{equation}
and $b(\cdot ,\cdot)$ satisfies the {\it inf-sup $(${\it Ladyzhenskaya-Babu\u{s}ka-Brezzi}$)$ condition},
\begin{align}
\label{inf-sup-sm}
\exists \, \beta >0\ \mbox{ such that } \ \inf _{q\in {\mathcal M}\setminus \{0\}}\sup_{v\in X\setminus \{0\}}\frac{b(v,q)}{\|v\|_X\|q\|_{\mathcal M}}\geq \beta .
\end{align}
Moreover, there exists a constant $C$ depending on $\beta $, $\lambda $ and the norm of $a(\cdot ,\cdot )$, such that the unique solution {$(u,p)\in {X}\times {\mathcal M}$} of \eqref{mixed-variational} satisfies the inequality
\begin{align}
\label{mixed-C}
{\|u\|_{X}+\|p\|_{{\mathcal M}}\leq C\left(\|f\|_{X^{*}}+\|g\|_{{\mathcal M}^{*}}\right).}
\end{align}
\end{thm}
\noindent In addition, we have (see \cite[Theorem A.56, Remark 2.7]{Ern-Gu}, \cite[Theorem 2.7]{Amrouche-2}).
\begin{lem}
\label{surj-inj-inf-sup}
Let $X,{\mathcal M}$ be reflexive Banach spaces. Let $b(\cdot ,\cdot):X\times {\mathcal M}\to {\mathbb R}$ be a bounded bilinear form. Let $B\!:\!X\!\to \!{\mathcal M}^{*}$ and $B^{*}\!:\!{\mathcal M}\!\to \!X^{*}$ be the operators defined by \eqref{B}, 
and let $V\!=\!{\rm{Ker}}\, B$. Then the following results are equivalent:
\begin{itemize}
\item[$(i)$]
There exists a constant $\beta >0$ such that $b(\cdot ,\cdot)$ satisfies 
condition \eqref{inf-sup-sm}.
\item[$(ii)$]
$B:{X/V}\to {\mathcal M}^{*}$ is an isomorphism and
$\|Bw\|_{{\mathcal M}^{*}}\!\geq \!\beta \|w\|_{X/V}$ for any $w\!\in \!X/V.$
\item[$(iii)$]
$B^{*}\!:\!{\mathcal M}\!\to \!V^\perp$ is an isomorphism and $\|B^{*}q\|_{X^{*}}\!\geq \!\beta \|q\|_{\mathcal M}$ for any $q\in {\mathcal M}.$
\end{itemize}
\end{lem}

\begin{rem}
\label{Hilbert-structure}
Let $X$ be a reflexive Banach space and {$V$ be a} closed subspace of $X$. If {a} bounded bilinear form $a(\cdot ,\cdot ):V\times V\to {\mathbb R}$ is {\it coercive} on $V$, i.e., there exists a constant $c_a>0$ such that
\begin{align}
\label{coercive}
a(w,w)\geq c_a\|w\|_X^2,\ \forall \, w\in V,
\end{align}
then the conditions \eqref{B-N-B1} are satisfied as well $($see, e.g., \cite[Lemma 2.8]{Ern-Gu}$)$.
\end{rem}
\noindent The next result known as the {\it Babu\u{s}ka-Brezzi theorem} is the version of Theorem \ref{B-N-B-new} for Hilbert spaces (see \cite{Babuska}, \cite[Theorems 0.1, 1.1, Corollary 1.2]{Brezzi}).
\begin{thm}
\label{B-B}
Let $X$ and ${\mathcal M}$ be two real Hilbert spaces. Let $a(\cdot ,\cdot):X\times X\to {\mathbb R}$ and $b(\cdot ,\cdot):X\times {\mathcal M}\to {\mathbb R}$ be bounded bilinear forms. Let $f\in X^{*}$ and $g\in {\mathcal M}^{*}$. Let $V$ be the subspace of $X$ defined by \eqref{V}. Assume that $a(\cdot ,\cdot ):V\times V\to {\mathbb R}$ is coercive and that $b(\cdot ,\cdot):X\times {\mathcal M}\to {\mathbb R}$ satisfies
the inf-sup condition \eqref{inf-sup-sm}.
{Then the variational problem \eqref{mixed-variational} is well-posed.}
\end{thm}

\section*{\bf Acknowledgements}
The research has been supported by the grant EP/M013545/1: "Mathematical Analysis of Boundary-Domain Integral Equations for Nonlinear PDEs" from the EPSRC, UK. 
Part of this work was done in April/May 2018, when M. Kohr visited the Department of Mathematics of the University of Toronto. She is grateful to the members of this department for their hospitality.

\end{document}